\documentclass{amsart}

\usepackage{esint}
\usepackage{amssymb}
\usepackage{tikz}
\usepackage{graphicx}
\usepackage{amsrefs}

\newtheorem{theorem}{Theorem}
\newtheorem{lemma}[theorem]{Lemma}

\theoremstyle{definition}

\newtheorem{remark}[theorem]{Remark}

\newtheorem{definition}[theorem]{Definition}

\newcommand{\eref}[1]{(\ref{e.#1})}
\newcommand{\tref}[1]{Theorem \ref{t.#1}}
\newcommand{\lref}[1]{Lemma \ref{l.#1}}
\newcommand{\fref}[1]{Figure \ref{f.#1}}
\newcommand{\sref}[1]{Section \ref{s.#1}}

\numberwithin{theorem}{section}
\numberwithin{equation}{section}

\newcommand{\Z}{\mathbb{Z}}
\newcommand{\R}{\mathbb{R}}

\newcommand{\C}{\mathbb{C}}

\newcommand{\ep}{\varepsilon}

\newcommand{\id}{\operatorname{1}}

\newcommand{\dist}{\operatorname{dist}}

\begin{document}

\title{A free boundary problem with facets}

\author{William M Feldman}

\author{Charles K Smart}

\begin{abstract}
  We study a free boundary problem on the lattice whose scaling limit is a harmonic free boundary problem with a discontinuous Hamiltonian.  We find an explicit formula for the Hamiltonian, prove the solutions are unique, and prove that the limiting free boundary has a facets in every rational direction.  Our choice of problem presents difficulties that require the development of a new uniqueness proof for certain free boundary problems.  The problem is motivated by physical experiments involving liquid drops on patterned solid surfaces.
\end{abstract}

\maketitle

\section{Introduction}

\subsection{Overview}

In this paper, we study a variational problem on the lattice $\Z^d$ whose scaling limit is a free boundary problem of the form
\begin{equation}\label{e.fb1}
  \begin{cases}
    L u = 0 & \mbox{in } \{ u > 0 \} \\
    H(\nabla u) = 1 & \mbox{on } \partial \{ u > 0 \},
  \end{cases}
\end{equation}
where $L$ is the Laplacian and $H$ is a lower semicontinuous Hamiltonian.  We study viscosity solutions of this problem for general $H$, and prove existence and uniqueness of solutions for certain boundary value problems.  We exactly compute our limiting Hamiltonian for the lattice problem, prove that it is not continuous, and show that the scaling limit has facets.

The main motivation for our study is to explain the appearance of facets in the contact line of liquid droplets wetting rough surfaces or spreading in a porous medium.  This phenomenon has been observed in physical experiments \cite{Raj-Adera-Enright-Wang,Kim-Zheng-Stone,pattern}.  While it is easy enough to construct a problem of the form \eref{fb1} with facets in the free boundary, we are able to derive such a problem as a scaling limit of a simple microscopic model for the liquid droplet problem.  Furthermore we find solutions which can be reliably obtained by a natural flow at the level of the microscopic problem, advancing the contact line from a small initial wetted set as was done in the experiments \cite{Raj-Adera-Enright-Wang}.

\subsection{A discrete free boundary problem}

Consider the following familiar variational problem.  Given an open set $U \subseteq \R^d$ whose boundary $\partial U$ is smooth and compact, compute the (local) minimizers of the energy
\begin{equation}
  \label{e.introJ}
  J[u] = \int_U \id_{\{ u > 0 \}}(x) +  |\nabla u(x)|^2 \,dx
\end{equation}
among the functions $u \in H^1_{loc}(\R^d)$ satisfying $u = 1$ on $\R^d \setminus U$.  This has a well-developed theory, see for example Caffarelli-Salsa \cite{Caffarelli-Salsa}, that leads to the free boundary problem
\begin{equation}
  \label{e.intropde}
  \begin{cases}
    L u = 0 & \mbox{in } U \cap \{ u > 0 \} \\
    |\nabla u| = 1 & \mbox{on } U \cap \partial \{ u > 0 \} \\
    u = 1 & \mbox{on } \R^d \setminus U \\
    u \geq 0 & \mbox{in } U,
  \end{cases}
\end{equation}
where $L$ denotes the Laplacian on $\R^d$.  While \eref{intropde} generally does not have a unique solution, there is always a least supersolution.

\begin{figure}
  \begin{tabular}{ll}
  \begin{tabular}{ll}
  \includegraphics[width=.24\textwidth]{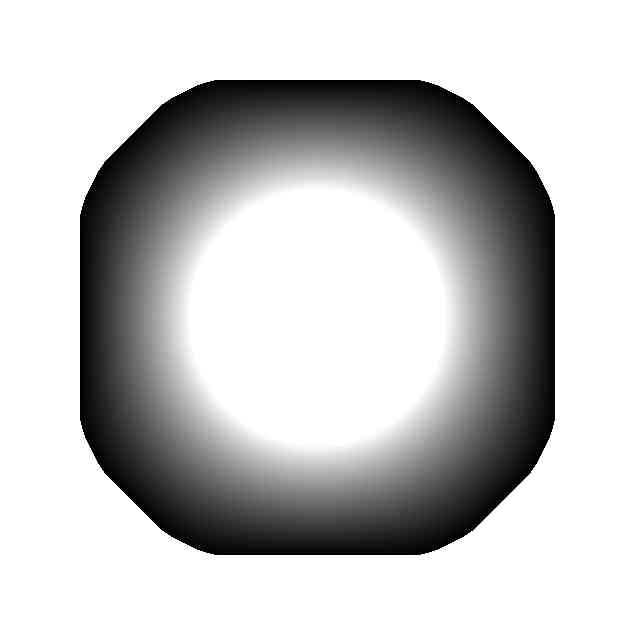} &
  \includegraphics[width=.24\textwidth]{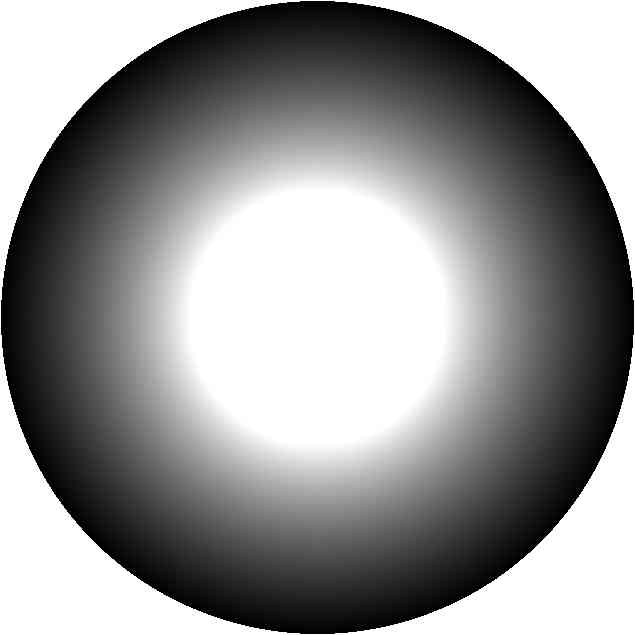} \\
  \includegraphics[width=.24\textwidth]{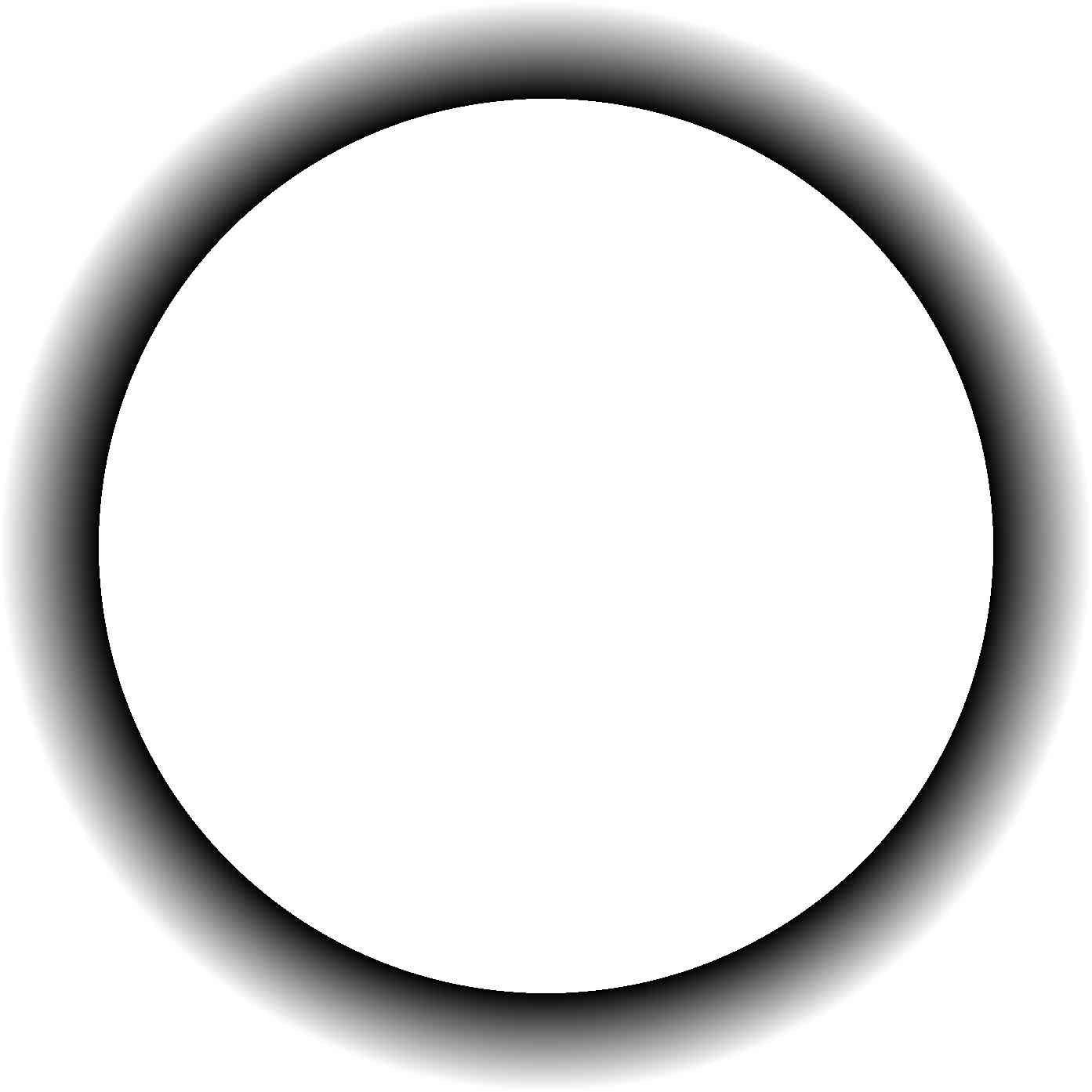} &
  \includegraphics[width=.24\textwidth]{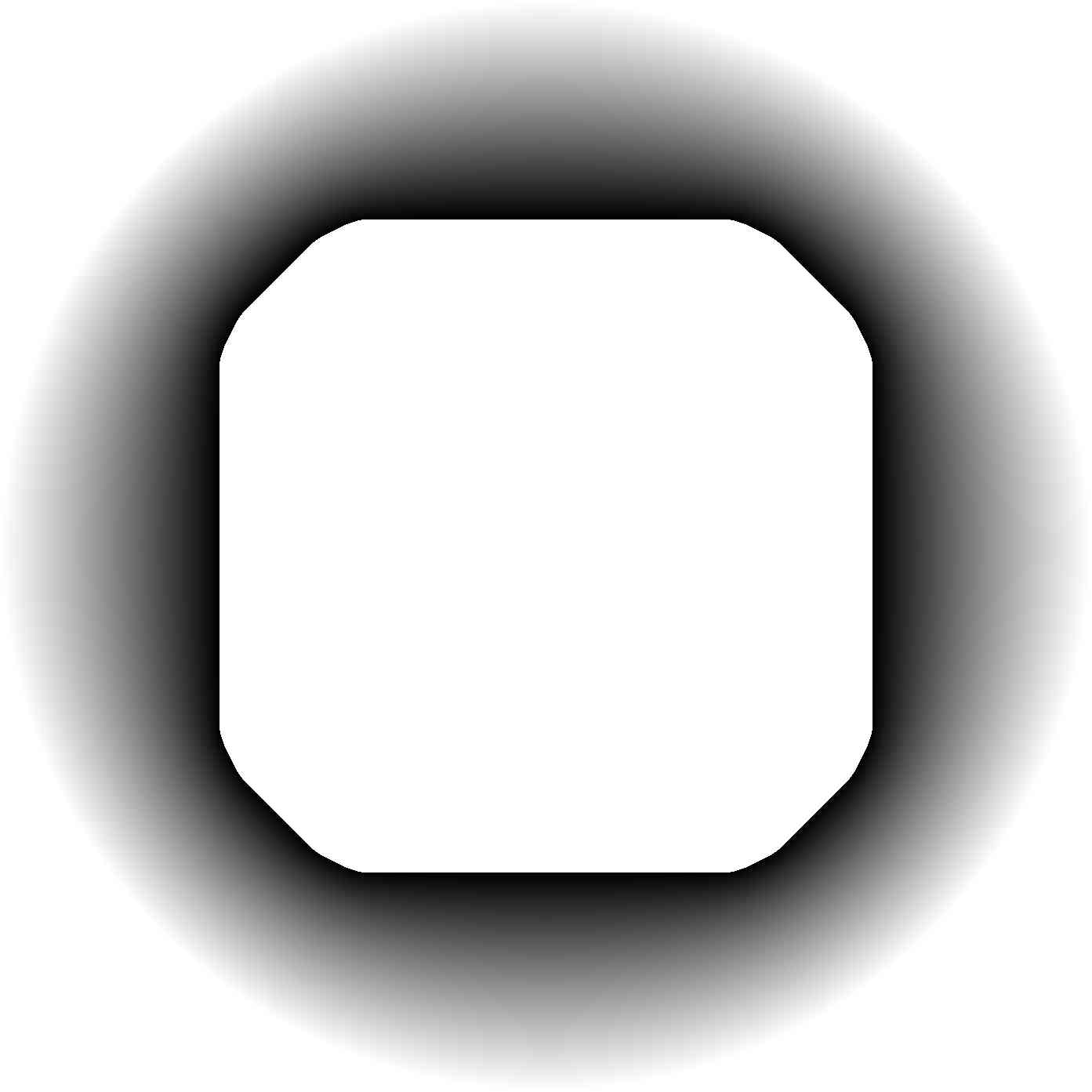} \\
  \end{tabular} &
  \begin{tabular}{l}
    \includegraphics[width=.4\textwidth]{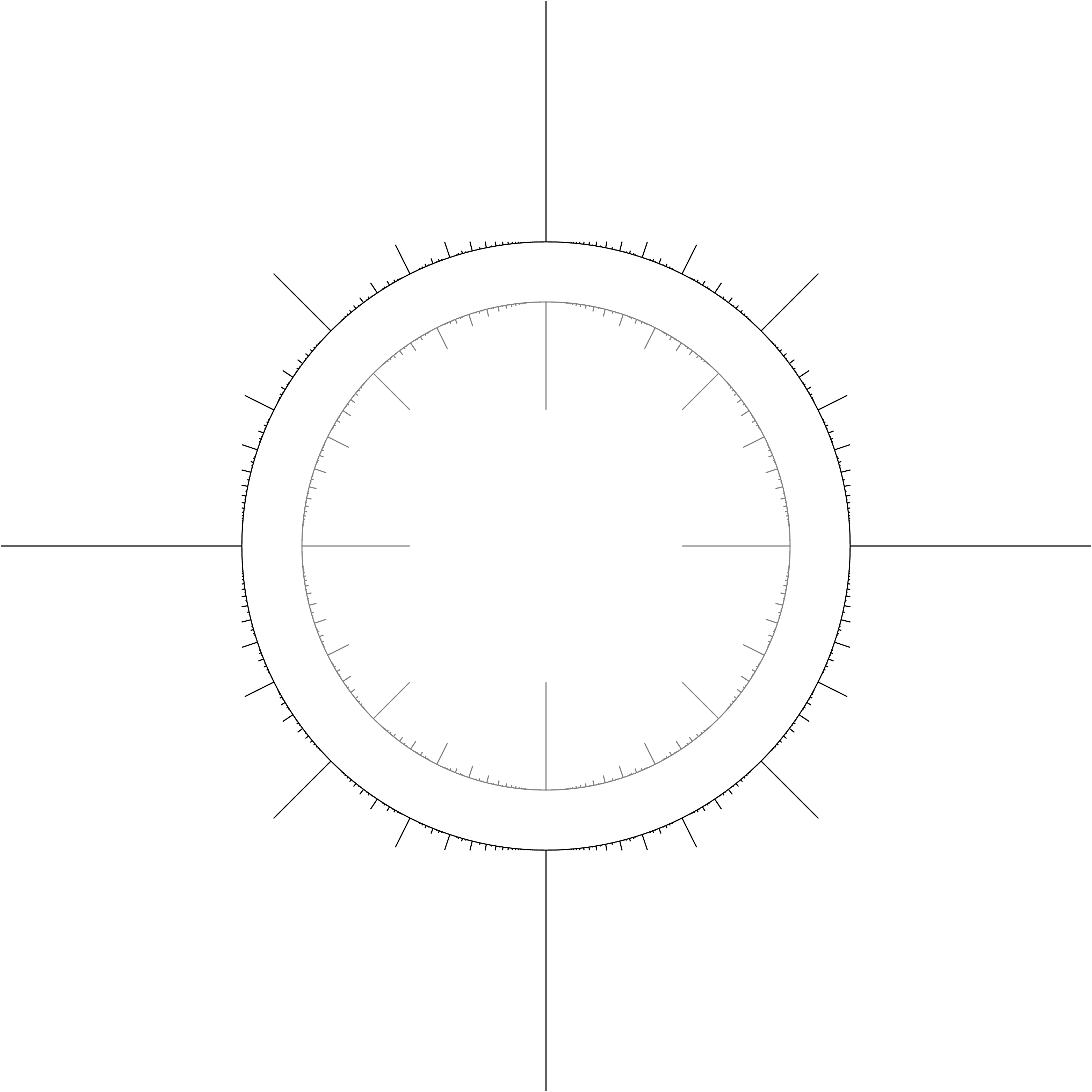}
  \end{tabular}
  \end{tabular}
  \caption{Top: $u_h$ and $u^h$ for $d = 2$, $U = B_1$, and $h = 2^7$.  Bottom: $u_h$ and $u^h$ for $d = 2$, $U = \R^d \setminus B_{4 e}$, and $h = 2^6$.  The free boundary is the black edge in all four images.  Right: the boundaries $\partial \{ H \leq 1 \}$ (black) and $\partial \{ \bar H \geq 1 \}$ (gray).}
  \label{f.intro2d}
\end{figure}

\begin{figure}
  \includegraphics[width=.48\textwidth]{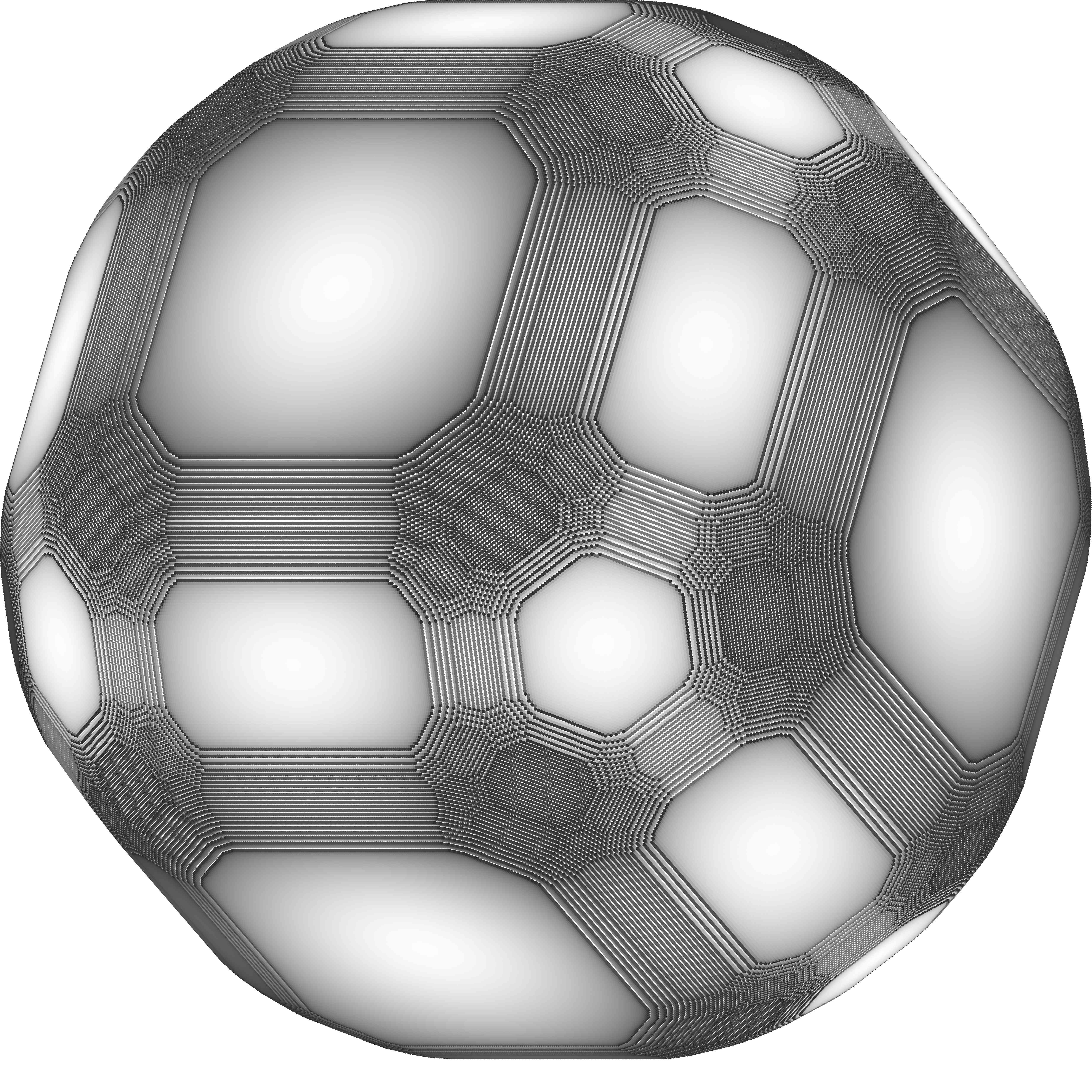}
  \hfill
  \includegraphics[width=.48\textwidth]{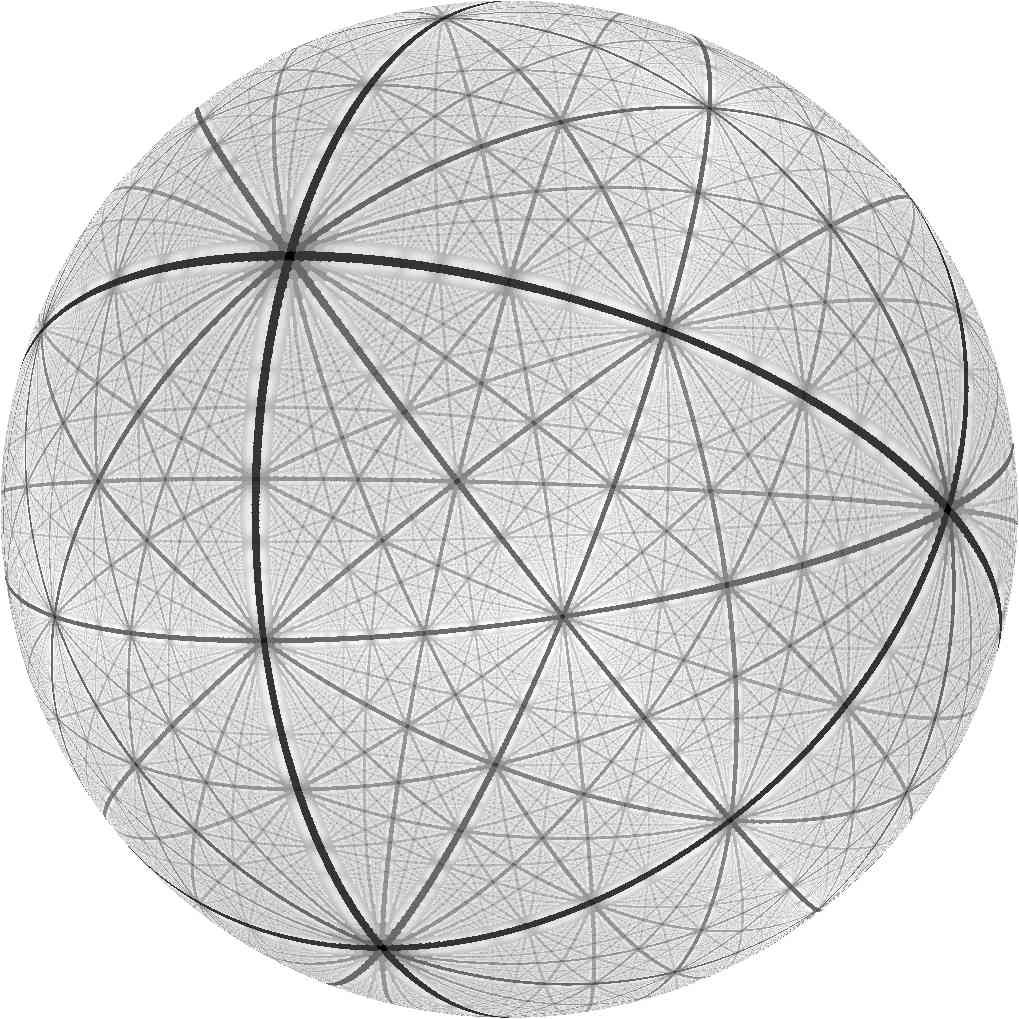}
  \caption{Left: the Laplacian $\Delta u_h$ on $\partial^+ \{ u_h > 0 \}$ for $d = 3$, $U = B_1$, and $h = 2^7$.  Right: the values of $H$ on the points $|p|^{-1} p$ for $p \in \Z^3$ with $\max_k |p_k| \leq 50$.}
  \label{f.intro3d}
\end{figure}

We consider a lattice discretization of the above variational problem.  Given a large scale $h > 1$, we study the critical points of the energy
\begin{equation}
  \label{e.introJh}
  J_h[u] = \sum_{x \in h U} \id_{\{ u > 0 \}}(x) + d \sum_{\substack{|x - y| = 1 \\ \{ x, y \} \cap h U \neq \emptyset}}  (u(x) - u(y))^2
\end{equation}
over functions $u : \Z^d \to \R$ that satisfy $u = h$ on $\Z^d \setminus h U$.  Note our choice of constants does not quite match with \eref{introJ}.  Computing the first variation, we see that local minimizers of $J_h$ should satisfy
\begin{equation}
  \label{e.introfde}
  \begin{cases}
    \Delta u = 0 & \mbox{on } h U \cap \{ u > 0 \} \\
    \Delta u \leq 1 & \mbox{on } h U \cap \partial^+ \{ u > 0 \} \\
    u \geq \tfrac{1}{2d} & \mbox{on } h U \cap \partial^- \{ u > 0 \} \\
    u = h & \mbox{on } \Z^d \setminus h U \\
    u \geq 0 & \mbox{on } \Z^d \cap h U,
  \end{cases}
\end{equation}
where $\Delta$ denotes the discrete Laplacian on $\Z^d$ and $\partial^+ X$ and $\partial^- X$ denote the outer and inner lattice boundary of a set $X \subseteq \Z^d$.

It is reasonable to expect that our discretization has a scaling limit described by the original continuum problem.  That is, if let $u_h : \Z^d \to \R$ denote the least supersolution of \eref{introfde}, then we might expect the rescalings
\begin{equation*}
  \bar u_h(x) = h^{-1} u_h(h x)
\end{equation*}
to converge to the least supersolution of \eref{intropde}.  The energy minimizers do converge to the energy minimizer of \eref{introJ} (with appropriate constants), this is a standard $\Gamma$-convergence argument.  The local minimizers, however, have a more complicated scaling limit.  Indeed, as we see in \fref{intro2d} and \fref{intro3d}, the least supersolutions $u_h$ are not radially symmetric for $U = B_1$.  Instead, we find ourselves in a situation analogous to that of Caffarelli-Lee \cite{Caffarelli-Lee} and Kim \cite{Kim}, who studied the local minimizers of
\begin{equation}\label{e.acenergy}
  \int_{U} Q(x/\ep)^2 \id_{u>0}(x) + |\nabla u(x)|^2 \,dx,
\end{equation}
where $Q : \R^d \to (0,\infty)$ is smooth and periodic and $\ep > 0$ is small.  The periodic structure creates preferred directions for the free boundary, and thus breaks radial symmetry in the homogenized limit.

For our stationary problem on the lattice, we are able to push a bit further than the past works.  In \tref{convergence} below, we identify the scaling limit of \eref{introfde} as
\begin{equation}
  \label{e.intropdeH}
  \begin{cases}
    L u = 0 & \mbox{in } U \cap \{ u > 0 \} \\
    H(\nabla u) = 1 & \mbox{on } U \cap \partial \{ u > 0 \} \\
    u = 1 & \mbox{on } \R^d \setminus U,
  \end{cases}
\end{equation}
where $H$ is a lower semicontinuous Hamiltonian that satisfies $H(tp) = t H(p)$ and $C^{-1} |p| \leq H(p) \leq C |p|$.  In \tref{ridges} we prove that the Hamiltonian $H$ is discontinuous at the slope $p$ whenever the slope $p$ satisfies one or more Diophantine relations.  Moreover, in \tref{facets}, we prove that the limiting free boundary has facets in every rational direction.

The case of the maximal non-trivial subsolution $u^h : \Z^d \to \R$ is dual to the minimal supersolution case.  Indeed, comparing the top and bottom rows of \fref{intro2d}, we see that $\partial \{ u_h > 0 \}$ has facets for the exterior while $\partial \{ u^h > 0 \}$ has facets for the interior.  The scaling limit of $u^h : \Z^d \to \R$ is captured by the upper semicontinuous Hamiltonian $\bar H(p) = 2d |p|^2 H(p)^{-1}$.  From the right of \fref{intro2d}, we see that there is a non-trivial pinning interval for all slopes.  Due to the duality, we only handle the case of the minimal supersolution in the paper below.

\subsection{Main results}

Our first result is an exact computation of $H(p)$.  
\begin{theorem}
  For $p \in \R^d$
  \begin{equation*}
    H(p)^2 = 2d \exp \left( \sum_{\substack{q \in \Z^d: \  p \cdot q = 0}} \hat S(q) \right) |p|^2,
  \end{equation*}
where $\hat{S}$ is the Fourier transform of $S(\theta) = \log(1+\frac{1}{d}\sum_{j=1}^d \cos \theta_j)$, a $2\pi \Z^d$-periodic function on $\R^d$.
\end{theorem}

In homogenization, it is rarely possible to find precise formulas like the above, and, when it is possible, it opens the possibility of a precise characterization of the shapes appearing in the scaling limit.

Our second result is the uniqueness of the scaling limit of the discrete problem for convex domains, which is consequence of \tref{facets} and \tref{convergence}.  Note that the terms ``least supersolution,'' ``facet,'' and ``rational direction'' are defined precisely in the body of the paper.

\begin{theorem}
  Let $U \subseteq \R^d$ be the complement of the closure of an open bounded and convex set.  There is a $u \in C^{0,1}_c(\R^d)$ such that, if $u_h : \Z^d \to \R$ denotes the least supersolution of \eref{introfde}, then $h^{-1} u_h(h x) \to u(x)$ uniformly in $x \in \R^d$ as $h \to \infty$.  Moreover, the support $\{ u > 0 \}$ is convex and has facets in every rational direction.
\end{theorem}

Our results for viscosity solutions of \eref{intropdeH} are stated below in \sref{viscosity}.  For lower semicontinuous Hamiltonians the viscosity subsolution condition is delicate.  We utilize three a priori different notions of viscosity subsolution, listed in increasing order of strength: weakened subsolution, modified subsolution, and (standard) subsolution.  These are defined precisely in \sref{viscosity}.  The weakened subsolution condition is easy to prove for the scaling limit, but is only sufficient for uniqueness in the convex case.  The modified subsolution condition, which we believe to be equivalent to the subsolution condition, is relatively difficult to prove for the scaling limit, and is needed for uniqueness in the non-convex case in two dimensions.  We only obtain the full subsolution condition a posteriori using uniqueness and Perron's method.

  We are also able to prove the existence of a unique solution for \eref{intropdeH} and the scaling limit of the $u_h$ without convexity in $d=2$.  Some geometric assumptions, such as strong star-shapedness, on $U$ are still needed to avoid the typical degenerate non-uniqueness which can occur for this type of problem.  The following is a consequence of \tref{convergence}, \tref{planestrictcomparison}, and \tref{ptfm}.

\begin{theorem}
Let $U \subseteq \R^d$ a regular open set.
\begin{enumerate}
\item[$(i)$] Any subsequential uniform limit $u$ of the rescalings $h^{-1} u_h(h x)$ is a supersolution and modified (and weakened) subsolution of \eref{intropdeH}.
\item[$(ii)$] If $d=2$ and $U$ is strongly star-shaped or $d \geq 3$ and $U$ is convex, then the limit is the unique solution of \eref{intropdeH}.
\end{enumerate}
\end{theorem}

\subsection{Water droplets on a rough surface} 

As mentioned earlier in the introduction, the first motivation for our work was to explain the formation of facets in the effective contact line of liquid drops wetting a patterned solid surface.  In a series of physical experiments, Raj-Adera-Enright-Wang~\cite{Raj-Adera-Enright-Wang} control the shape of a liquid droplet by creating a surface patterned with periodic arrays of micro-pillars.  For an expanding droplet the contact line is pinned sooner when it is parallel to a lattice direction. This leads the formation of facets in the contact line.  Raj et al.~\cite{Raj-Adera-Enright-Wang} are able to create surfaces so that the steady state wetted set appears to be polygonal, squares, hexagons, octagons and others.  The goal of this engineering research is essentially the inverse version of the problem we consider.  That is, can one find a planar graph for which the limit of local minimizers of the analogue of \eref{introJh} has a particular convex shape?  The article \cite{Raj-Adera-Enright-Wang} suggests that this may be possible.

We were interested to explain the results of \cite{Raj-Adera-Enright-Wang} via a relatively simple mathematical model.  The key feature we were interested to capture was the existence of facets in the free boundary of either the minimal super-solution or the maximal sub-solution of the Euler-Lagrange equation.  Note that since the water droplets in \cite{Raj-Adera-Enright-Wang} were achieved by an advancing contact line they should arise as a minimal super-solution, or at least the minimal supersolution above a certain obstacle.  In general the minimal super-solution and maximal sub-solution are of physical importance since they are the only local minimizers which can be found reliably by the flow without prior knowledge of the solution, either advancing from a small initial contact set or receding from a large initial contact set.

To model the physics accurately, one would ideally study the full capillarity energy.  Global energy minimizers of the capillarity energy were studied by Caffarelli-Mellet~\cite{Caffarelli-Mellet0} for flat patterned surfaces, and for rough surfaces by Alberti-DeSimone~\cite{Alberti-DeSimone} and more recently quantified by the first author and Kim in \cite{Feldman-Kim}.  As with our discrete model, the global energy minimizers turn out to have axial symmetry in the limit as the length scale of the roughness $\to 0$.  Thus the formation of facets in the contact line is a property of local minimizers, or, possibly, the length scales in the physical experiments are not sufficiently separated for a homogenization type argument.  After their work \cite{Caffarelli-Mellet0}, Caffarelli-Mellet~\cite{Caffarelli-Mellet} also studied local minimizers, showing the existence of a sequence of local minimizers which converged to a non-axially-symmetric limit.  Characterizing the pinning interval and the limiting free boundary problem for minimal super-solutions (or maximal sub-solutions) of the full capillarity problem seems to be a difficult problem. 

One possible simplification of the full capillarity model is the Alt-Caffarelli type energy \eref{acenergy} with an oscillating coefficient.  This problem was studied by Caffarelli-Lee~\cite{Caffarelli-Lee}, who established the existence of facets in convex free boundaries for certain local minimizers.  The dynamic version of this problem was studied by Kim \cite{Kim}, who proved that the limiting normal velocity can have non-trivial pinning intervals, which can lead to facets formation for carefully chosen initial data.  While the scaling limit of our problem is of the same form as the homogenized limit of \eref{acenergy}, a significant part our our work is devoted to proving convexity and facet formation of the free boundary of the limiting minimal supersolution of the natural boundary value problem \eref{introfde}.  The difficulties with the continuum models, and the appealing similarity between the finite solutions in \fref{intro2d} and the droplets observed in \cite{Raj-Adera-Enright-Wang}, were one motivation for our study of \eref{introfde}.  

\subsection{The boundary sandpile}

There is a connection between our model and the boundary sandpile considered by Aleksanyan-Shahgholian \cite{Aleksanyan-Shahgholian,Aleksanyan-Shahgholian-2}.  Indeed, our discrete model first appears in this work.  The divisible sandpile (see for example Levine \cite{Levine}) is a deterministic diffusion process on the lattice in which configurations $\rho : \Z^d \to [0,\infty)$ evolve by toppling.  Each site has some fixed capacity and, at each time step, evenly divides and sends its excess to its neighbors.  When all sites have capacity $1$, the evolution is described by:
\begin{equation*}
  \rho_{k+1} = \rho_k + (2d)^{-1} \Delta \max \{ 0, \rho_k - 1 \}
\end{equation*}
In the boundary sandpile model, once a site topples, its capacity is set to zero.  That is, one also keeps track of the set of toppled vertices $D_k \subseteq \Z^d$ and evolves according to:
\begin{equation*}
  D_{k+1} = D_k \cup \{ \rho_k > 1 \} \quad \mbox{and} \quad \rho_{k+1} = \rho_k + (2d)^{-1} \Delta \max \{ 0, \rho_k - \id_{\{\Z^d \setminus D_k\}}\}
\end{equation*}
If $\rho_0 : \Z^d \to [0,\infty)$ has finite support, then the evolution stabilizes.  That is, $\rho_\infty = \lim_{k \to \infty} \rho_k$ exists.  Moreover, the limit can be computed as $\rho_\infty = \rho_0 + \Delta u$, where $u : \Z^d \to \R$ is the least function satisfying
\begin{equation*}
  u \geq 0 \quad \mbox{and} \quad \rho_0 + \Delta u \leq \id_{\{ u = 0 \}}.
\end{equation*}
This is quite similar to computing the least supersolution of \eref{introfde}.  Indeed, it is simply the Poisson version.  The connection between the dynamics described here and our discrete problem is made clear below in \lref{flow}.

\subsection{Outline}

In \sref{hamiltonian}, we study the Hamiltonian used to describe the boundary condition of the limit of local minimizers of \eref{introJh}.  We use a contour integration to find a nice formula for the Hamiltonian.  In \sref{viscosity}, we develop the viscosity solution theory of \eref{intropdeH} for general lower semicontinuous Hamiltonians.  Here we prove uniqueness of weak solutions and the formation of facets on the free boundary when the data is convex.  In \sref{scaling}, we prove that local minimizers of \eref{introJh} converge, under the appropriate scaling and with suitable data, to the unique solution of \eref{intropdeH}.  This is essentially standard given the work of the previous sections.  Finally, in \sref{nonconvex}, we explore removing the convexity hypothesis from our results in \sref{viscosity}.  We prove that this is possible in dimension two and lay groundwork for higher dimensions. 

\subsection{Source code}

The source code used to generate the figures in this article is included in the arXiv submission.

\subsection{Acknowledgments}

The first author was partially supported by the National Science Foundation.  The second author was partially supported by the National Science Foundation and the Alfred P Sloan foundation.  Both authors benefited from conversations with Hayk Aleksanyan and Henrik Shahgholian.

\section{The Hamiltonian}
\label{s.hamiltonian}

\subsection{Essential properties}

In this subsection, we define the limiting Hamiltonian $H$ and study its basic properties.  These properties suffice to capture the scaling limit of \eref{introJh}.  The later subsections study the fine structure of $H$, which is required to prove the limit has facets.

Our definition is inspired by the condition $\Delta u \leq \id_{\{ u = 0 \}}$ that is implied by \eref{introfde}.  If, in the blow up limit, we obtain a solution $\Z^d$ that looks roughly like $u(x) = \max \{ 0, p \cdot x \}$, then the discrete Laplacian on the boundary should be at most one.  To define the Hamiltonian, we reverse this thinking, and measure the Laplacian on the boundary of a half space solution.

\begin{definition}
  If $p \in \R^d \setminus \{0\}$, then $H(p) = \Delta u(0)$, where $u : \Z^d \to \R$ solves
  \begin{equation}
    \label{e.halfspace}
    \begin{cases}
      \Delta u(x) = 0 \quad \mbox{for } p \cdot x > 0 \\
      u(x) = 0 \quad \mbox{for } p \cdot x \leq 0 \\
      \sup_{p \cdot x > 0} |u(x) - p \cdot x| < \infty.
    \end{cases}
  \end{equation}
\end{definition}

First we state a maximum principle in half-spaces.

\begin{lemma}
  \label{l.halfball}
  There is a constant $C > 0$ such that, if $p \in \R^d \setminus \{0\}$, $R > r \geq 1$, and $u : \Z^d \to \R$ satisfies $\Delta u = 0$ in
  \begin{equation*}
    D = \{ q \in \Z^d : p \cdot q > 0 \mbox{ and } |q| < R \},
  \end{equation*}
  then
  \begin{equation*}
    \max_{B_r \cap D} |u| \leq \max_{B_{R/2} \cap \partial^+ D} |u| + C \frac{r}{R} \max_{\partial^+ D} |u|.
  \end{equation*}
\end{lemma}

\begin{proof}
  This is a standard barrier argument.
\end{proof}

The uniqueness of \eref{halfspace} follows from \lref{halfball}, and more generally it implies the following maximum principle: if $u$ is subharmonic and bounded in $\{x \cdot p >0\}$ then
\[ \sup_{p \cdot x >0} u(x) = \sup_{\partial^+\{p \cdot x >0\}} u(x).\]
  The maximum principle implies the following easy estimates.

\begin{lemma}
  \label{l.halfspace}
  For $p \in \R^d$, the solution $u : \Z^d \to \R$ of \eref{halfspace} satisfies
  \begin{equation}
    \label{e.barriers}
    \max \{ 0, p \cdot x \} \leq u(x) \leq \max \{ 0, p \cdot x \} + \| p \|_\infty \quad \mbox{for } x \in \Z^d.
  \end{equation}
  Moreover, $\Delta u \leq \Delta u(0) \id_{\{ u = 0 \}}$.
\end{lemma}

\begin{proof}
  The lower and upper bounds in \eref{halfspace} are, respectively, a subsolution and a supersolution of \eref{halfspace}.  The error term $\| p \|_\infty$ guarantees that the right-hand side is non-negative on the boundary $\partial^+ \{ q \in \Z^d : p \cdot q > 0 \}$.  
  
  It is more convenient to prove $\Delta u \leq \Delta u(0) \id_{\{ u = 0 \}}$ later in Lemma~\ref{gform}.
\end{proof}

When combined with a compactness argument, the easy estimates from \lref{halfspace} tell us all that we need to know about $H$ to capture the scaling limit of our discrete problem.

\begin{lemma}
  For $p \in \R^d$ and $t > 0$, we have
  \begin{equation}
    \label{e.Hhomogeneous}
    H(t p) = t H(p),
  \end{equation}
  \begin{equation}
    \label{e.Hbounded}
    C^{-1} |p| \leq H(p) \leq C |p|,
  \end{equation}
  and
  \begin{equation}
    \label{e.Hlowersemicontinuous}
    \liminf_{q \to p} H(q) \geq H(p).
  \end{equation}
\end{lemma}

\begin{proof}
  The homogeneity \eref{Hhomogeneous} is immediate from the definitions.  The bounds \eref{Hbounded} follow from the bounds \eref{barriers}.  It remains to prove the lower semicontinuity \eref{Hlowersemicontinuous}.  It is enough to show that, if $p_n \to p$ and $H(p_n) \to s$ as $n \to \infty$, then $H(p) \leq s$.  Let $u_n : \Z^d \to \R$ solve \eref{halfspace} for $p_n$.  On account of the bounds \eref{barriers}, we may pass to a subsequence to assume that $u_n(x)$ has a limit $v(x)$ for all $x \in \Z^d$. Observe that $v : \Z^d \to \R$ satisfies
  \begin{equation*}
    \begin{cases}
      v(0) = 0 \\
      \Delta v(0) \leq s \\
      \Delta v(x) = 0 \mbox{ when } v(x) > 0 \\      
      \max\{ 0, p \cdot x \} \leq v(x) \leq \max\{ 0, p \cdot x \} + \| p \|_\infty.
    \end{cases}
  \end{equation*}
  The maximum principle gives $u \leq v$ for the solution $u : \Z^d \to \R$ of \eref{halfspace} for $p$.  Since $u(0) = v(0)$ and $u \leq v$, we have $\Delta u(0) \leq \Delta v(0) \leq s$ and $H(p) \leq s$.
\end{proof}

\subsection{Rational slopes}

In order to understand the fine structure of our Hamiltonian $H$, we compute its values for rational slopes.  We rewrite our definition of $H$ several times, making it easier to compute at each iteration.  We first change \eref{halfspace} into a one dimensional problem.

\begin{lemma}\label{gform}
  If $p \in \Z^d$ and $\gcd(p_1, ..., p_d) = 1$, then
  \begin{equation}
    \label{e.Hhitting}
    H(p) = |p|^2 \lim_{k \to \infty} g(k),
  \end{equation}
  where $g : \Z \to \R$ is the unique solution of
  \begin{equation}
    \label{e.hittingmeasure}
    \begin{cases}
      \Delta_p g(k) = 0 \quad \mbox{for } k > 0 \\
      g(0) = 1 \\
      g(k) = 0 \quad \mbox{for } k < 0 \\
      \sup_{k > 0} |g(k)| < \infty
    \end{cases}
  \end{equation}
  and
  \begin{equation*}
    \Delta_p g(k) = \sum_{j = 1}^d (g(k+p_j) + g(k-p_j) - 2 g(k)).
  \end{equation*}
\end{lemma}

\begin{proof}
  We begin by observing that the solution $u$ of \eref{halfspace} satisfies $u(x+y) = u(x)$ whenever $x, y \in \Z^d$ and $p \cdot y = 0$.  Thus, if we quotient by the lattice $\Lambda = \{ q \in \Z^d : p \cdot q = 0 \}$, then we obtain a one dimensional problem.  That is, there is a function $f : \Z \to \R$ such that
  \begin{equation*}
    u(x) = f(p \cdot x)
  \end{equation*}
  and
  \begin{equation*}
    \begin{aligned}
      \Delta u(x)
      & = \sum_{i = 1}^d (u(x+e_i) + u(x-e_i) - 2 u(x)) \\
      & = \sum_{i = 1}^d (f(p \cdot x + p_i) + f(p \cdot x - p_i) - f(p \cdot x)) \\
      & = \Delta_p f(p \cdot x).
    \end{aligned}
  \end{equation*}
  Since $p$ is irreducible, the operator $\Delta_p$ is the Laplacian of connected and translation invariant graph on $\Z$.  In particular, $f : \Z \to \R$ is the unique solution of
  \begin{equation*}
    \begin{cases}
      \Delta_p f(k) = 0 \quad \mbox{for } k > 0 \\
      f(k) = 0 \quad \mbox{for } k \leq 0 \\
      \sup_{k > 0} |f(k) - k| < \infty.
    \end{cases}
  \end{equation*}
  We observe that
  \begin{equation*}
    \sum_{k \leq 0} \Delta_p f(k) = |p|^2.
  \end{equation*}
  Since $\Delta_p f(k)$ is distributed on $k \leq 0$ according to $\Delta_p$-harmonic measure from $+\infty$, we obtain $H(p) = \Delta u(0) = \Delta_p f(0) = |p|^2 \lim_{k \to \infty} g(k)$.
  
  At this point we can also check that $\Delta u \leq \Delta u(0)$.  Consider $f(k+1) - f(k)$, bounded and $\Delta_p$-harmonic on $\Z_+$ and $f \geq 0$ on the $\Delta_p$-boundary of $\Z_+$.  Thus by maximum principle $f$ is increasing, and so also is $\Delta_p f(k)$ for $k \leq 0$.
\end{proof}

We study the roots of the characteristic polynomial $\lambda^n Q(\lambda)$ of $\Delta_p$.

\begin{lemma}
  If $p \in \Z^d$ and $\gcd(p_1, ..., p_d) = 1$, then
  \begin{equation*}
    Q(\lambda) = \sum_{k = 1}^d (\lambda^{p_k} + \lambda^{-p_k} - 2)
  \end{equation*}
  can be written
  \begin{equation}
    \label{e.qfactor}
    Q(\lambda) = m \lambda^{-n} \prod_{k=1}^n (\lambda - \lambda_k)(\lambda - \lambda_k^{-1}),
  \end{equation}
  where $n = \max_k |p_k|$, $m = \# \{ k : |p_k| = n \}$, and the roots $\lambda_k \in \C$ satisfy
  \begin{equation}
    \label{e.rootbox}
    1 = \lambda_1 < |\lambda_2| \leq \cdots \leq |\lambda_n|,
    \quad \mbox{and} \quad
    \lambda_2, ..., \lambda_n \notin [1,\infty).
  \end{equation}
\end{lemma}

\begin{proof}
  The existence of a factorization of the form \eref{qfactor} follows from the observation that $m^{-1} \lambda^n Q(\lambda)$ is a monic and palindromic polynomial.  Since
  \begin{equation*}
    Q(e^t) = \sum_{j=1}^d (2 \cosh(p_j t) - 2),
  \end{equation*}
  we see that $1$ is the only positive real root.  Since
  \begin{equation*}
    Q(e^{it}) = \sum_{j = 1}^d (2 \cos(p_j t) - 2)
  \end{equation*}
  and $\gcd(p_1, ..., p_d) = 1$, we see that $1$ is the only root of unit modulus.
\end{proof}

We write $H(p)$ in terms of the roots of $Q$.

\begin{lemma}
  If $p \in \Z^d$ and $\gcd(p_1, ..., p_d) = 1$, then
  \begin{equation}
    \label{e.Hproduct}
    H(p)^2 = |p|^2 m \prod_{k = 2}^n (-\lambda_k),
  \end{equation}
  where $\lambda_k \in \C$, $n$, and $m$ are as in \eref{qfactor} and \eref{rootbox}.
\end{lemma}

\begin{proof}
  Observe that $\lambda^n Q(\lambda)$ is the characteristic polynomial of $\Delta_p$.  Let us temporarily assume that
  \begin{equation}
    \label{e.distinct}
    \lambda_2, ..., \lambda_n \mbox{ are distinct,}
  \end{equation}
  so that we may solve \eref{hittingmeasure} by making the ansatz
  \begin{equation*}
    g(k) = \sum_{j = 1}^n \alpha_j \lambda_j^{-k} \quad \mbox{for } k > -n.
  \end{equation*}
  Since the $\lambda_j$ are roots of the characteristic polynomial, we obtain
  \begin{equation*}
    \Delta_p g(k) = 0 \mbox{ for } k > 0.
  \end{equation*}
  Since the $\lambda_j$ have modulus at least one, we obtain
  \begin{equation*}
    \sup_{k > 0} |g(k)| < \infty.
  \end{equation*}
  To meet the boundary conditions, we must have
  \begin{equation*}
    \sum_{j = 1}^n \lambda_j^{k-1} \alpha_j = \delta_{1k} \quad \mbox{for } k = 1, ..., n.
  \end{equation*}
  The Vandermonde
  \begin{equation*}
    V_{kj} = \lambda_j^{k-1}
  \end{equation*}
  is invertible by \eref{distinct}.  We compute the inverse as follows.  Since the polynomials
  \begin{equation*}
    P_k(\lambda) = V^{-1}_{kj} \lambda^{j-1}
  \end{equation*}
  satisfy
  \begin{equation*}
    P_k(\lambda_j) = \sum_l V_{kl}^{-1} \lambda_j^{l-1} = \sum_l V_{kl}^{-1} V_{lj} = \id_{kj},
  \end{equation*}
  we recognize them as the Lagrange polynomials
  \begin{equation*}
    P_k(\lambda) = \prod_{j \neq k} \frac{\lambda - \lambda_j}{\lambda_k - \lambda_j}.
  \end{equation*}
  Thus $V^{-1}_{kj}$ is the coefficient of $\lambda^{j-1}$ in $P_k(\lambda)$.  In particular,
  \begin{equation*}
    \lim_{k \to \infty} g(k) = \alpha_1 = V^{-1}_{11} = \operatorname{coeff}(P_1,1) = \prod_{j = 2}^n \frac{-\lambda_j}{1-\lambda_j}.
  \end{equation*}
  To simplify this product, we compute
  \begin{equation*}
    |p|^2 = \tfrac12 Q''(1) = m \prod_{k=2}^n (1 - \lambda_k) (1 - \lambda_k^{-1}) = m \alpha_1^{-2} \prod_{k=2}^n (-\lambda_k).
  \end{equation*}
  We conclude that
  \begin{equation*}
    H(p)^2 = |p|^4 \lim_{k \to \infty} g(k)^2 = |p|^4 \alpha_1^2 = |p|^2 m \prod_{j=2}^n (-\lambda_j).
  \end{equation*}
  Since the result is continuous in the $\lambda_j$, we can eliminate the assumption \eref{distinct} by a density argument.  We perturb the characteristic polynomial, which in turn perturbs the operator, and use \eref{rootbox} to pass to limits.
\end{proof}

Appealing to complex analysis, we transform the expression \eref{Hproduct} of $H(p)$ in terms of the roots of $Q(\lambda)$ into an oscillatory integral.

\begin{lemma}
  If $p \in \Z^d$ and $\gcd(p_1, ..., p_d) = 1$, then
  \begin{equation}
    \label{e.Hintegral}
    H(p)^2 = |p|^2 2d \exp \left( \frac{1}{2\pi} \int_0^{2\pi} \log \left( 1 - \frac{1}{d} \sum_{j=1}^d \cos(p_k t) \right) \,dt \right).
  \end{equation}
\end{lemma}

\begin{proof}
  Recalling \eref{qfactor}, observe that
  \begin{equation*}
    Q(e^{iz}) = \sum_{k = 1}^d 2 (\cos(p_k z) - 1) = m \prod_{k=1}^n (2 \cos z - \lambda_k - \lambda_k^{-1}).
  \end{equation*}
  Removing the zero at $1$, we consider the holomorphic function
  \begin{equation*}
    h(z) = \sum_{k = 1}^d \frac{1 - \cos(p_k z)}{1 - \cos(z)} = m \prod_{k=2}^n (2 \cos z - \lambda_k - \lambda_k^{-1}).
  \end{equation*}
  For all sufficiently large $L > 0$, the restrictions \eref{rootbox} imply that the zeros of $h$ in rectangle $\Omega = (0,2\pi) \times (0,L)$ are exactly $i \log(\lambda_2), ..., i \log(\lambda_n)$ and that $h$ does not vanish on $\partial \Omega$.  By the Residue Theorem,
  \begin{equation*}
    \sum_{j = 2}^n \log(\lambda_j) = -\frac{1}{2 \pi} \int_{\partial \Omega} \frac{z h'(z)}{h(z)} \,dz.
  \end{equation*}
  Since $h$ does not vanish on $\partial \Omega$, there is a holomorphic $w : \partial \Omega \setminus \{ 0 \} \to \C$ such that
  \begin{equation*}
    h = e^w \quad \mbox{on } \partial \Omega \setminus \{ 0 \}.
  \end{equation*}
  Integrating by parts, we obtain
  \begin{equation*}
    -\frac{1}{2 \pi} \int_{\partial \Omega} \frac{z h'(z)}{h(z)} \,dz = -\frac{1}{2\pi} \int_{\partial \Omega \setminus \{ 0 \}} z w'(z) \,dz = \frac{1}{2\pi} \int_{\partial \Omega \setminus \{ 0 \}} w(z) \,dz.
  \end{equation*}
  We evaluate the right-hand side by estimating $w$ along the four sides of the rectangle.  Since $h$ is real on $[0,2\pi]$ and $[2\pi, 2\pi + L i]$, we may assume that
  \begin{equation*}
    w = \log h \quad \mbox{on } (0,2\pi] \cup [2 \pi, 2 \pi + L i].
  \end{equation*}
  For the top side $[2\pi + Li, Li]$, we use the asymptotic estimate
  \begin{equation*}
    h(t + L i) = m e^{(n-1) (L - i t)} + O(e^{(n-2) L}).
  \end{equation*}
  Since $w(2\pi + Li)$ is real, we conclude that
  \begin{equation*}
    w(t + L i) = \log m + (n-1) L + (n-1)(2 \pi - t) i + O(e^{-L}).
  \end{equation*}
  Since $h$ is also real on $[Li, 0]$, we must have
  \begin{equation*}
    w(si) = w(2 \pi + si) + (n-1) 2 \pi i.
  \end{equation*}
  We are now ready to compute the integral.  We have
  \begin{equation*}
    \frac{1}{2 \pi} \int_{[0,2\pi]} w(z) \,dz = \frac{1}{2\pi} \int_0^{2\pi} \log h(t) \,dt,
  \end{equation*}
  \begin{equation*}
    \frac{1}{2 \pi} \int_{[2\pi,2\pi+Li]} w(z) \,dz + \frac{1}{2 \pi} \int_{[Li,0]} w(z) \,dz = (n-1) L,
  \end{equation*}
  and
  \begin{equation*}
    \frac{1}{2\pi} \int_{[2\pi+Li,Li]} w(z) \,dz = - \log m - (n-1) L - (n-1) \pi i + O(e^{-L}).
  \end{equation*}
Since the value of the contour integral is independent of $L$ for large $L>0$, the $O(e^{-L})$ term is actually equal to zero, and it follows that
  \begin{equation*}
    \log \left( m \prod_{j=2}^n (-\lambda_j) \right) = \frac{1}{2\pi} \int_0^{2\pi} \log h(t) \,dt.
  \end{equation*}
    Plug in the definition of $h$ and use that,
  \begin{equation*}
    \frac{1}{2\pi} \int_0^{2\pi} \log(1 - \cos t)  \ dt = - \log 2.
  \end{equation*}
  Conclude by appealing to \eref{Hproduct}.
\end{proof}

In preparation for our final formula for $H(p)$, we study the Fourier coefficients of an abstraction of the integrand from \eref{Hintegral}.

\begin{lemma}
  \label{l.Sdecay}
  The $2\pi\Z^d$-periodic function
  \begin{equation*}
    S(\theta) = \log \left( 1 + \frac{1}{d} \sum_{k=1}^d \cos(\theta_k) \right)
  \end{equation*}
  has Fourier coefficients $\hat S : \Z^d \to \R$ that satisfy
  \begin{equation}
    \label{e.Sdecay}
    0 \leq - \hat S(q) \leq C (1 + |q|)^{-d} \log(2 + |q|).
  \end{equation}
\end{lemma}

\begin{proof}
  Step 1.  The lower bound follows from the estimates
  \begin{equation*}
    |S|(\theta) \leq C (1 + |\log \theta|) \quad \mbox{and} \quad |\nabla^d S|(\theta) \leq C |\theta|^{-d} \quad \mbox{for } \theta \in [-\pi,\pi]^d \setminus \{ 0 \}
  \end{equation*}
  and a standard integration by parts argument.  For $\delta \in (0,1)$ to be determined, select a function $\eta \in C^\infty(\R^d / (2 \pi \Z^d))$ satisfying $\eta \equiv 0$ in $B_\delta$, $\eta \equiv 1$ in $[-\pi,\pi]^d \setminus B_{2\delta}$, and $|\nabla^k \eta|(\theta) \leq C |\theta|^{-k}$ for $\theta \in B_{2\delta} \setminus B_\delta$ and $k = 1, ..., d$.  Using the identity
  \begin{equation*}
    \left( \frac{iq}{|q|^2} \cdot \nabla \right) e^{i q \cdot \theta} = - e^{i q \cdot \theta},
  \end{equation*}
  compute
  \begin{equation*}
    \begin{aligned}
      \hat S(-q) & = \fint_{[-\pi,\pi]^d} e^{i q \cdot \theta} S(\theta) \,d\theta \\
      & = \fint_{[-\pi,\pi]^d} e^{i q \cdot \theta} \left[ \left( \frac{iq}{|q|^2} \cdot \nabla \right)^d \eta S \right](\theta) \,d\theta + \fint_{[\pi,\pi]^d} e^{iq \cdot \theta} [(1 - \eta) S](\theta) \,d\theta \\
      & = O\left( \int_{[-\pi,\pi]^d \setminus B_\delta} |q|^{-d} |\theta|^{-d} \,d\theta \right) + O\left( \int_{B_{2\delta}} 1+|\log \theta| \,d\theta \right) \\
      & = O(|q|^{-d} |\log \delta|) + O(\delta^d (1 + |\log \delta|)).
    \end{aligned}
  \end{equation*}
  Now select $\delta = (2+|q|)^{-1}$.

  Step 2.  The upper bound follows by series expansion.  For $\theta \in [-\pi,\pi]^d \setminus \{ 0 \}$,
  \begin{equation*}
    \begin{aligned}
      \log \left( 1 - \frac{1}{d} \sum_{j = 1}^d \cos(\theta_j) \right)
      & = - \sum_{k = 1}^\infty \frac{1}{k} \left( \frac{1}{d} \sum_{j = 1}^d \cos(\theta_j) \right)^k \\
      & = - \sum_{k = 1}^\infty \frac{1}{k(2d)^k} \left( \sum_{j = 1}^d e^{i\theta_j} + e^{-i\theta_j} \right)^k.
    \end{aligned}
  \end{equation*}
  Observe that the coefficient of $e^{i q \cdot \theta}$ in any partial sum is real and non-positive.  Conclude by dominated convergence.
\end{proof}

We now arrive at the simplest formulation of $H(p)$. 

\begin{lemma}
  \label{l.Hrational}
  If $p \in \Z^d \setminus \{0\}$ then
  \begin{equation}
    \label{e.Hformula}
    H(p)^2 = 2d \exp \left( \sum_{\substack{q \in \Z^d \\ p \cdot q = 0}} \hat S(q) \right) |p|^2.
  \end{equation}
\end{lemma}

\begin{proof}
  Use the Fourier expansion of $S$ to formally compute
  \begin{equation*}
    \begin{aligned}
      & \hspace{-2em} \frac{1}{2\pi} \int_0^{2\pi} \log\left(1 - \frac{1}{d} \sum_{k=1}^d \cos(p_k t)\right) \,dt \\
      & = \frac{1}{2\pi} \int_0^{2\pi} S(p t) \,dt \\
      & = \sum_{q \in \Z^d} \frac{1}{2\pi} \int_0^{2\pi} \hat S(q) e^{i q \cdot p t} \,dt \\
      & = \sum_{p \cdot q = 0} \hat S(q).
    \end{aligned}
  \end{equation*}
  Since the sublattice $\{ q \in \Z^d : p \cdot q = 0 \}$ is contained in a subspace of $\R^d$ of codimension $1$, this formal computation is made rigorous by the decay estimate \lref{Sdecay}.  We conclude by appealing to \eref{Hintegral}.
\end{proof}

\subsection{Irrational slopes}

We now extend our formula \eref{Hformula} to irrational $p$.  Observe that the sum in \eref{Hformula} for $H(p)$ is over all the Diophantine conditions that $p$ satisfies.  In particular, the sum at least makes sense for any $p \in \R^d$.  We define the lattice of Diophantine conditions satisfied by $p$:

\begin{definition}
For $p \in \R^d \setminus \{ 0 \}$, let $\Lambda_p = \{ q \in \Z^d : p \cdot q = 0 \}.$
\end{definition}

The map $p \in \Lambda_p$ is, of course, not continuous in $p$.  However, it is upper semicontinuous with respect to inclusion, as the following lemma shows.

\begin{lemma}
  \label{l.approximate}
  For every unit vector $p \in \R^d$ and $\ep > 0$ there is a $q \in \Z^d$ such $\gcd(q_1, ..., q_d) = 1$, $\Lambda_p \subseteq \Lambda_q$, and $|p - |q|^{-1} q| < \ep$.
\end{lemma}

\begin{proof}
  Define $\Lambda_p^\perp = \{ q \in \Z^d : p \cdot q = 0 \mbox{ for } p \in \Lambda_v \}$ and observe that the quotient $\Z^d / (\Lambda_p + \Lambda_p^\perp)$ is finite.  By the pigeonhole principle, there is a $C_p > 0$ such that there are infinitely many $q \in \Lambda_p + \Lambda_p^\perp$ with $|p - |q|^{-1} q| \leq C_p |q|^{-2}$.  Since $p \perp \Lambda_p$, we may as well take $q \in \Lambda_p^\perp$ so that $\Lambda_p \subseteq \Lambda_q$.
\end{proof}

Using the uniqueness of \eref{halfspace} from \lref{halfball} and the upper semicontinuity of $p \mapsto \Lambda_p$, we are able to extend our formula to arbitrary $p \in \R^d$.

\begin{theorem}
  \label{t.Hformula}
  The formula \eref{Hformula} holds for arbitary $p \in \R^d$.
\end{theorem}

\begin{proof}
  We may assume that $p \in \R^d$ is a unit vector.  Let $u_q$ denote the solution of \eref{halfspace} for slope $q \in \R^d$.  Using \lref{approximate}, choose a sequence of $q_n \in \Z^d$ such that $\Lambda_p \subseteq \Lambda_{q_n}$ and $|q_n|^{-1} q_n \to p$.  Observe that, for any $x \in \Z^d$, that $p \cdot x$ and $q_n \cdot x$ have the same sign for all sufficiently large $n$.  Thus
  \begin{equation*}
    \lim_{n \to \infty} 2d \exp\left(\sum_{\Lambda_{q_n}} \hat S \right) = 2d \exp \left(\sum_{\Lambda_p} \hat S \right).
  \end{equation*}
  Moreover, for any $R > 0$,
  \begin{equation*}
    B_R \cap \{ x \in \Z^d : p \cdot x > 0 \} = B_R \cap \{ x \in \Z^d : q_n \cdot x > 0 \}
  \end{equation*}
  holds for sufficiently large $n$.  Letting $D_R$ denote the above set, we observe that
  \begin{equation*}
    u_p = u_{q_n} = 0 \quad \mbox{on } B_{R/2} \cap \partial^+ D_R
  \end{equation*}
  and, using \eref{barriers}, that
  \begin{equation*}
    |u_p - u_{q_n}| \leq C \quad \mbox{on } \partial^+ D_R
  \end{equation*}
  holds for all sufficiently large $n$.   Applying \lref{halfball}, we see that $\max_{B_2} |u_p - u_{q_n}| \leq C R^{-1}$ for all large $n$.  Since $R > 0$ was arbitrary, we obtain $\lim_{n \to \infty} \Delta u_{q_n}(0) = \Delta u_p(0)$ and the theorem.
\end{proof}

\subsection{Fine structure}

From the formula \eref{Hformula} for $H$, we see immediately that $H$ is lower semicontinuous.  In fact, we can prove something much finer.  Looking at the plot of $H$ on the sphere in dimension $3$ in \fref{intro3d}, we see that $H$ has ``valleys'' on the sphere that correspond to Diophantine conditions.  In higher dimensions, there are Diophantine conditions that make ``valleys'' of arbitrary codimension in $H$ on the sphere.  Each point on the sphere lies in a ``valley'' given by a lattice of Diophantine conditions.  The sparser the lattice, the less deep the ``valley.''  Generic points of the sphere satisfy no conditions, and therefore the generic points all have the same maximal value.  We encode all of this in the following theorem.

\begin{theorem}
  \label{t.ridges}
  For every $p \in \R^d \setminus \{ 0 \}$, we have
  \begin{equation*}
    \lim_{\substack{q \to p \\ \Lambda_p \subseteq \Lambda_q}} H(q) = H(p)
    \quad \mbox{and} \quad
    \liminf_{\substack{q \to p \\ \Lambda_p \nsubseteq \Lambda_q}} H(q) > H(p).
  \end{equation*}
  Moreover, for $\mathcal H^{d-1}$-almost every unit vector $p \in \R^d$,
  \begin{equation*}
    H(p)^2 = 2d \exp(\hat S(0)).
  \end{equation*}
\end{theorem}

\begin{proof}
  Using $p \neq 0$, we see that, for every $R > 0$, there is a $\delta > 0$ such that $|p - q| < \delta$ implies $B_R \cap \Lambda_q \subseteq \Lambda_p$.  This proves the first limit.  For the second limit, choose a finite generating set $X \subseteq \Lambda_p$ and let $\ep = \min_X (- \hat S)$.  Using \eref{Sdecay}, we may choose $R > 0$ such that
  \begin{equation*}
    X \subseteq B_R
  \end{equation*}
  and
  \begin{equation*}
    0 \leq \sup_{q \in \R^d \setminus \{ 0 \}} \sum_{\Lambda_q \setminus B_R} (- \hat S) \leq \ep/2.
  \end{equation*}
  Observe that, if $|p - q|$ small and $\Lambda_p \nsubseteq \Lambda_q$, then $X \nsubseteq \Lambda_q$ and $B_R \cap \Lambda_q \subseteq \Lambda_p$.  Thus
  \begin{equation*}
    \sum_{\Lambda_q} (- \hat S) \leq \ep/2 + \sum_{B_R \cap \Lambda_q} (- \hat S) \leq - \ep / 2 + \sum_{B_R \cap \Lambda_p} (- \hat S) \leq - \ep/2 + \sum_{\Lambda_p} (-\hat S).
  \end{equation*}
  For the final statement, we observe that, for $x \in \Z^d$, the set $\{ p \in \partial B_1 : x \in \Lambda_p \}$ has zero $(d-1)$-dimensional Hausdorff measure.  Since $\Z^d$ is countable, we see that $\Lambda_p = \{ 0 \}$ for $\mathcal H^{d-1}$-almost every $p \in \partial B_1$.
\end{proof}

In $d=2$, the structure is simpler.

\begin{theorem}
  \label{t.spikes}
  For any $p \in \R^2 \setminus \{ 0 \}$, $\lim_{q \to p} H(q)^2 = 4 \exp(\hat S(0)) |p|^2$.  Moreover, $H(p)^2 < 4 \exp(\hat S(0)) |p|^2$ if and only if $p$ has is a multiple of an element of $\Z^2 \setminus \{ 0 \}$.
\end{theorem}

\begin{proof}
  The limit follows from the observation that, for every $R > 0$, there is a $\delta > 0$ such that $0 < |q - p| < \delta$ implies $\Lambda_q \cap B_R = \{ 0 \}$.  The second follows from the observation that $\Lambda_p \neq \{ 0 \}$ if and only if $p$ has a rational direction.
\end{proof}

\begin{remark}
  When $d = 2$, we can show that the inradius of $\{ H \leq 1 \}$ is $\exp(-\tfrac{2}{\pi} K)$, where $K \approx 0.91597$ is Catalan's constant.  We do not know an explicit formula when $d \geq 3$.
\end{remark}

\section{Viscosity Solutions}
\label{s.viscosity}

\subsection{Basic notions}

In this section we study the exterior problem
\begin{equation}
  \label{e.exterior}
  \begin{cases}
    L u = 0 & \mbox{in } \{ u > 0 \} \setminus \overline W \\
    H(\nabla u) = 1 & \mbox{on } \partial \{ u > 0 \} \setminus \overline W \\
    u = 1 & \mbox{on } \overline W
  \end{cases}
\end{equation}
for a $W \subseteq \R^d$ that is open, bounded, and inner regular.  We assume familiarity with the viscosity interpretation of \eref{exterior} for $H(p) = |p|$ given, for example, in Caffarelli-Salsa \cite{Caffarelli-Salsa}.  We interpret \eref{exterior} in the viscosity sense, with the necessary modifications to treat a discontinuous Hamiltonian.  Except for the last subsection, all results of this section apply to any $H$ satisfying \eref{Hhomogeneous}, \eref{Hbounded}, and \eref{Hlowersemicontinuous}.

The viscosity interpretation of an equation moves the derivatives onto test functions using the maximum principle.  At first glance, this means that a function should be a viscosity supersolution if it enjoys comparison from below by strict classical subsolutions.  Similarly, a function should be a viscosity subsolution if it enjoys comparison from above by strict classical supersolutions.  However, we need to be careful about the meaning of classical solution.  In order for viscosity solutions to be stable under uniform convergence, the classical solutions need to be stable under perturbations.  Since $H$ is merely lower semicontinuous, we must therefore restrict our notion of classical supersolution.  This is reflected in the definitions below.

\begin{definition}
  If $X$ is a topological space, then $USC(X)$ and $LSC(X)$ denote the upper and lower semicontinuous functions on $X$.  We say $u \in USC(X)$ touches $v \in LSC(X)$ from below (and $v$ touches $u$ from above) in $X$ if $u \geq v$ and the contact set $\{ u = v \}$ is non-empty and compact.
\end{definition}

\begin{definition}
  A supersolution of \eref{exterior} is a function $u \in LSC(\R^d)$ that is compactly supported, satisfies $u \geq \id_{\overline W}$, and such that, whenever $\varphi \in C^\infty(\R^d)$ touches $u$ from below in $\R^d \setminus \overline W$, there is a contact point $x$ such that either
  \begin{equation*}
    L \varphi(x) \leq 0
  \end{equation*}
  or
  \begin{equation*}
    \varphi(x) = 0 \quad \mbox{and} \quad H(\nabla \varphi(x)) \leq 1.
  \end{equation*}
\end{definition}

\begin{definition}
  A subsolution of \eref{exterior} in a function $u \in USC(\R^d)$ that is compactly supported, satisfies $u \leq \id_{\overline W}$, and such that, whenever $\varphi \in C^\infty(U)$ touches $u$ from above in $\overline{\{ u > 0 \}} \setminus \overline W$, there is a contact point $x$ such that either
  \begin{equation*}
    L \varphi(x) \geq 0
  \end{equation*}
  or $\varphi(x) = 0$ and
  \begin{equation*}
    \limsup_{y \to x} H(\nabla \varphi(y)) \geq 1.
  \end{equation*}
\end{definition}

In addition to the (essentially standard) definition of supersolution and subsolution given above, we also need a weakened form of subsolution.  We obtain this notion by further restricting the class of classical supersolutions used as test functions.  This is necessary to obtain solutions whose support is convex.  This is also necessary to capture the scaling limit.

\begin{definition}
  A weakened subsolution of \eref{exterior} in a function $u \in USC(\R^d)$ that is compactly supported, satisfies $u \leq \id_{\overline W}$ and such that, whenever $\varphi \in C^\infty(\R^d)$ touches $u$ from above in $\overline{\{ u > 0 \}} \setminus \overline W$, there is a contact point $x$ such that either
  \begin{equation*}
    L \varphi(x) \geq 0,
  \end{equation*}
  or $\varphi(x) =0$ and either
  \begin{equation*}
    H(\nabla \varphi(x)) \geq 1,
  \end{equation*}
  or
  \begin{equation*}
    \nabla \varphi(U) \mbox{ contains two linearly independent slopes.}
  \end{equation*}
\end{definition}

\begin{remark}
  In the above definitions, we can make the inequalities $L \varphi(x) \geq (\leq ) 0$ strict without changing the meaning.  This is a consequence of the strong maximum principle and the homogeneity \eref{Hhomogeneous}.
\end{remark}

It is useful to define a special class of functions solving part of \eref{exterior}.

\begin{definition}
  If $\overline W \subseteq U \subseteq \R^d$ is open and bounded, then let $w_{W,U} \in H^1(\R^d)$ denote the unique solution of
  \begin{equation*}
    \begin{cases}
      L w = 0 & \mbox{in } U \setminus \overline W \\
      w = 1 & \mbox{on } \overline W \\
      w = 0 & \mbox{on } \R^d \setminus \overline U.
    \end{cases}
  \end{equation*}
  We say that $w_{W,U}$ is a supersolution, subsolution, or weakened subsolution if it is, respectively, a supersolution, subsolution, or weakened subsolution of \eref{exterior}.
\end{definition}

\subsection{Existence}

We obtain existence of a solution of \eref{exterior} by Perron's method.  This requires three ingredients: stability under uniform convergence, barrier functions, and some regularity to stay in the continuous category.  We assemble these ingredients below.

\begin{lemma}
  \label{l.stability}
  The notions of supersolution, subsolution, and weakened subsolution are stable under uniform convergence.
\end{lemma}

\begin{proof}
  This is a standard feature of viscosity solutions.  Here it is important that, for any smooth $\varphi$, the function $x \mapsto H(\nabla \varphi(x))$ is lower semicontinuous and the function $x \mapsto \limsup_{y \to x} H(\nabla \varphi(y))$ is upper semicontinuous.
\end{proof}

Harmonic lifting preserves the sub/super-solution property.
\begin{lemma}
  \label{l.harmonize}
  If $u \in LSC(\R^d)$ is a supersolution of \eref{exterior} for some $W \subseteq \R^d$ open, bounded, and inner regular, then $w_{W,\{u>0\}}$ is a supersolution.  If $u \in USC(\R^d)$ is a (weakened) subsolution of \eref{exterior} for some $W \subseteq \R^d$ open, bounded, and inner regular, then $w_{W,\{u>0\}}$ is a (weakened) subsolution.

\end{lemma}

\begin{proof}
  We prove the supersolution case.  The (weakened) subsolution case is similar.  The maximum principle implies $0 \leq w = w_{W, \{u > 0\}} \leq u$.  Since $W$ is inner regular, $w$ is continuous up to $\partial W$.  Since $u$ is lower semicontinuous, $w$ is lower semicontinuous up to $\partial \{ u > 0 \}$.  In particular, $w \in LSC(\R^d)$ has compact support and satisfies $w \geq \id_{\overline W}$.  If $\varphi \in C^\infty(\R^d \setminus \overline W)$ touches $w$ from below, then there is a contact point $x$ such that either $L \varphi(x) \leq 0$ or $\varphi(x) = 0$ and $\nabla \varphi(x) \neq 0$.  In the latter case, we see that $\varphi$ touches $u$ from below, and conclude $H(\nabla \varphi(x)) \leq 1$.
\end{proof}

\begin{lemma}
  \label{l.lipschitz}
  If $W \subseteq \R^d$ is open bounded and inner regular, $\overline W \subseteq U \subseteq \R^d$ is open and bounded, and $w_{W,U}$ is a supersolution, then $\| w_{W,U} \|_{C^{0,1}(\R^d)} \leq C_W$.
\end{lemma}

\begin{proof}
  Using \eref{Hbounded}, we see that $w_{W,U}$ also a supersolution for the simpler Hamiltonian $\tilde H(p) = C^{-1} |p|$.  The Lipschitz estimate of Alt-Caffarelli \cite{Alt-Caffarelli} now applies.
\end{proof}

Our existence theorem has two parts: we can either solve \eref{exterior} outright or we can find a weakened solution whose support is convex.

\begin{theorem}
  \label{t.existence}
  Suppose $0 \in W \subseteq \R^d$ is open, bounded, and inner regular.  There is a $u \in C_c(\R^d)$ that is a supersolution and subsolution of \eref{exterior}.  There is a $v \in C_c(\R^d)$ with $\{ v > 0 \}$ convex that is a supersolution and weak subsolution \eref{exterior}.
\end{theorem}

\begin{proof}
  The Lipschitz regularity allows us to carry out Perron's method in the continuous category.  That is, by \lref{stability}, \lref{harmonize}, and \lref{lipschitz}, the point-wise infimum of all supersolutions is a supersolution $u \in C^{0,1}_c(\R^d)$.  The minimality and continuity then imply that $u$ is a subsolution.

  For the convex case, we need to show that the infimum of supersolutions with convex support is a supersolution with convex support.  The key observation is the following.  Suppose that $w \in LSC(\R^d)$ is a supersolution of \eref{exterior} with $V = \{ w > 0 \}$ convex.  By the strong maximum principle, $V = \{ w_{W,V} > 0 \}$.  By the Lipschitz estimate in \lref{lipschitz}, we have $W + B_\delta \subseteq V$ for $\delta = \delta_W > 0$.  Now, if $B_r(x) \subseteq V$ is arbitrary and $V_{r,x}$ denotes the convex hull of $B_r(x) \cup (W + B_\delta)$, then we have $w \geq w_{W,V} \geq w_{W,V_{r,x}}$.  Since the lower bound depends only on $W$ and $B_r(x)$, we see that the infimum $v$ of all supersolutions $w$ with convex support has $\{v >0 \} = \cap_{w} \{ w>0\}$, and hence $v$ is a supersolution with convex support.  By minimality, the weakened subsolution property holds.
\end{proof}

\subsection{Convex comparison lemmas}

This subsection contains the technical lemmas required to prove uniqueness.

\begin{definition}
  An open set $U \subseteq \R^d$ has an inner tangent ball of radius $r > 0$ at $x \in \partial U$ if there is a $y \in U$ such that $B_r(y) \subseteq U$ and $x \in \partial B_r(y)$.  An open set $U \subseteq \R^d$ is $r$-inner regular if $U$ has an inner tangent ball of radius $r$ at every point in $\partial U$.  We define outer regularity by replacing $U$ with $\R^d \setminus \overline U$.
\end{definition}

Our first two lemmas concern the regularity of Lipschitz harmonic functions at regular points of the boundary of their support.

\begin{lemma}
  \label{l.C32}
  If $u \in C^{0,1}(B_2)$ is harmonic in $\{ u > 0 \}$ and
  \begin{equation*}
    B_1(1,0,...,0) \subseteq \{ u > 0 \} \subseteq \R^d \setminus \overline B_1(-1,0,...,0),
  \end{equation*}
  then there is an $s > 0$ such that
  \begin{equation*}
    |u(x) - s x_1| \leq C \| u \|_{C^{0,1}(B_2)} |x|^{4/3}
    \quad \mbox{for }
    x \in B_1 \cap \{ u > 0 \}.
  \end{equation*}
\end{lemma}

\begin{proof}
  We may assume $\| u \|_{C^{0,1}} = 1$.  Fix $r \in (0,1)$.  Let $v \in C(\overline B_r)$ be harmonic in $B_r \cap \{ x_1 > 0 \}$ and satisfy $v = 0$ on $\overline B_r \cap \{ x_1 \leq 0 \}$ and $v = u$ on $\partial B_R \cap \{ x_1 > 0 \}$.  Using the Lipschitz bound and the tangent balls, we obtain
  \begin{equation*}
    \| v - u \|_{L^\infty(B_r)} \leq C r^2
  \end{equation*}
  and
  \begin{equation*}
    \| v \|_{L^\infty(B_r)} \leq C r.
  \end{equation*}
  Since $w(x) = \textup{sgn}(x_1)v(|x_1|,x_2,...,x_d)$ is harmonic in $B_r$, we obtain 
  \begin{equation*}
    |w(x) - \nabla w(0) \cdot x| \leq C r^{-2} \| w \|_{L^\infty(B_r)} |x|^2 \quad \mbox{for } x \in B_{r/2}.
  \end{equation*}
  Since $\nabla w(0) = (s,0,...,0)$ with $s \geq 0$, we have thus proved the following estimate: For every $r \in (0,1)$, there is an $s \geq 0$ such that
  \begin{equation*}
    |u(x) - s x_1| \leq C (r^2 + r^{-1} |x|^2) \quad \mbox{for } x \in \{ u > 0 \} \cap B_{r/2}.
  \end{equation*}
  Using the scaling $r \sim |x|^{2/3}$, a standard iteration gives the quantitative first-order expansion.  The Hopf Lemma implies $s > 0$.
\end{proof}

\begin{lemma}
  \label{l.C32above}
  If $u, p, s$ are as in \lref{C32} and $\{ u > 0 \} \subseteq \{ x_1 > 0 \}$, then
  \begin{equation*}
    u(x) \leq (s + C \| u \|_{C^{0,1}(B_2)} |x|^{1/3}) x_1
    \quad \mbox{for }
    x \in B_{1/2} \cap \{ u > 0 \}.
  \end{equation*}
\end{lemma}

\begin{proof}
  We may assume $\| u \|_{C^{0,1}(B_2)} = 1$.  For $r \in (0,1)$, let $u_r \in C(B_r)$ be harmonic in $B_r \cap \{ x_1 > 0 \}$ and satisfy $u_r = 0$ on $B_r \cap \{ x_1 \leq 0 \}$ and $u_r = \max \{ 0, u \}$ on $\partial B_r \cap \{ x_1 > 0 \}$.  Observe that $u \leq u_r$ in $B_r \cap \{ x_1 > 0 \}$ and that $u_r(x) = \max \{ 0, u(x) \} \leq s x_1 + C r^{4/3}$ for $x \in \partial B_r \cap \{ x_1 > 0 \}$.  For $\alpha > 0$, we see that
  \begin{equation*}
    \begin{cases}
      u_r(x) \leq (s + \alpha r^{1/3}) x_1 - r^{4/3} & \mbox{for } x \in \partial B_r \cap \{ x_1 > C r \alpha^{-1} \} \\
      u_r(x) \leq (s + \alpha r^{1/3}) x_1 + C r^{4/3} & \mbox{for } x \in \partial B_r \cap \{ 0 < x_1 < C r \alpha^{-1} \} \\
      u_r(x) \leq (s + \alpha r^{1/3}) x_1 & \mbox{for } x \in B_r \cap \partial \{ x_1 > 0 \}.
    \end{cases}
  \end{equation*}
  By explicit computation with the Poisson kernel one can see that
  \[ \sup_{x \in B_{r/2} \cap \{ x_1 >0\}} \omega_x(\partial B_r \cap \{ 0 < x_1 < C r \alpha^{-1} \})/\omega_x( \partial B_r \cap \{ x_1 > C r \alpha^{-1} \}) \leq C\alpha^{-1}\]
  where $\omega_x$ is the harmonic measure of $B{r} \cap \{ x_1 >0\}$ from a point $x \in B_{r} \cap \{ x_1 > 0 \}$.  In particular, for $\alpha$ sufficiently large, we obtain $u_r(x) \leq (s + \alpha r^{1/3}) x_1$ for all $x \in B_{r/2} \cap \{ x_1 > 0 \}$.  Since this estimate is independent of $r \in (0,1)$, the lemma follows.
\end{proof}

The next lemma is a low-dimensional analogue of the Alexandroff-Bakelman-Pucci estimate.  The idea is that, near a saddle point of an upper semicontinuous function, one can touch from above by parabolas while maintaining some control the slopes at the touching point.

\begin{lemma}
  \label{l.convextouch}
  Suppose $X \subseteq \R^d$ bounded and convex, $\alpha > 0$, $g : \R^d \to \R$ upper semicontinuous, $g(y) < \alpha |y-x|^2$ for $x \in X$ and $y \in \R^d$, and $g(x) = 0$ for some $x \in \overline X$.  There are $y_n, p_n, y, p \in \R^d$ such that $y_n \to y$, $|p_n|^{-1} p_n \to p$, $y \in \bar X$, $g(y) = 0$, $p \cdot y = \max_{x \in \bar X} p \cdot x$, and $g(y) \leq g(y_n) + p_n \cdot (y - y_n) + 2 \alpha |y - y_n|^2$.
\end{lemma}

\begin{proof}
  After a change of coordinates, we may assume that $X$ is a relatively open convex and bounded subset of the hyperplane $\{ x_{k+1} = \cdots = x_d = 0 \}$ for some $k = 1, ..., d$ and that $0 \in X$.  In particular, we have $\tau \overline X \subseteq X$ for all $\tau \in (0,1)$.  After scaling, we may assume that $\alpha = 1$.  Consider the function
  \begin{equation*}
    h(\tau) = \min_{\substack{y \in \R^d \\ x \in \overline X}} \left\{ 2 |y - \tau x|^2 - g(y) \right\} \quad \mbox{for } \tau \in (0,1].
  \end{equation*}
  Observe that $h(\tau)$ is positive, Lipschitz, and non-increasing for $\tau \in (0,1)$ and that $h(1) = 0$.  In particular, $h'(\tau_n)$ exists and is negative for a sequence of $\tau_n \in (0,1)$ with $\tau_n \to 1$.  Using the definition of $h$ to compute the derivative, we may choose $y_n \in \R^d$ and $x_n \in \overline X$ such that
  \begin{equation*}
    4 x_n \cdot (y_n - \tau_n x_n) > 0
  \end{equation*}
  and
  \begin{equation*}
    2 |y_n - \tau_n x_n|^2 - g(y_n) = h(\tau_n).
  \end{equation*}
  The optimality condition tells us that
  \begin{equation*}
    g(y) \leq g(y_n) + p_n \cdot (y - y_n) + 2 |y - y_n|^2
  \end{equation*}
  for
  \begin{equation*}
    p_n = 4 (y_n - \tau_n x_n).
  \end{equation*}
  The optimality condition also tells us that
  \begin{equation*}
    p_n \cdot y_n > p_n \cdot \tau_n x_n = \max_{x \in \tau_n \overline X} p_n \cdot x.
  \end{equation*}
  Note that we must have $g(y_n) \to 0$ and $|y_n - \tau_n x_n|^2 \to 0$.  In particular, we may pass to a subsequence to obtain $y_n \to y \in \overline X \cap \{ g = 0 \}$.
\end{proof}

The next lemma is the key technical result required to prove uniqueness.  Our geometric picture is the following: when a surface touches a convex set on a facet, then either the surface covers the entire facet or it peels away.  When it peels away, we can find regular points of the surface arbitrarily close to the facet whose normal lies in the direction of the peeling.  

\begin{lemma}
  \label{l.convexcomparison1}
  Suppose $W \subseteq V \subseteq U \subseteq \R^d$ are open and bounded, $W$ and $V$ are convex and inner regular, $U$ is outer regular, $\overline W \subseteq V$, and $\partial V \cap \partial U \neq \emptyset$.  For any $\ep > 0$, we can find a point $x \in \partial U$ where $U$ has an inner tangent ball such that $p = \nabla w_{W,U}(x) \neq 0$ and $|\nabla w_{W,V}| \leq (1+\ep)|p|$ on $\overline V_p = \{ x \in \overline V : p \cdot x = \min_{y \in \overline V} p \cdot y \}$.
\end{lemma}

\begin{proof}
  Throughout the proof, we let $C$ denote a positive constant that may depend on $U, V, W$ and differ in each instance.  Write $u = w_{W,U}$ and $v = w_{W,V}$.  Choose a supporting hyperplane $\Lambda$ with unit inward normal $\nu$ for $V$ at some point of $\partial V \cap \partial U$.  Let $X = \Lambda \cap \partial V$ and $Y = \Lambda \cap \partial V \cap \partial U$.  Without loss we can assume that $0 \in \Lambda$.

  Since $V$ is convex and inner regular, \lref{C32} implies that $\nabla v$ extends continuously to $\partial V$.  Observe that $|\nabla v|^{-1} \nabla v = \nu$ on $X$.  Moreover, by Caffarelli-Spruck \cite{Caffarelli-Spruck}, we know that $\{ v > t \}$ is convex for all $t \in (0,1)$.  Since $\nabla v$ is continuous up to the boundary, this implies $|\nabla v|$ is concave on $X$.
  
  \lref{C32} also implies that $\nabla u$ exists and $|\nabla u|^{-1} \nabla u = \nu$ on $Y$.  The maximum principle implies $|\nabla v| \leq |\nabla u|$ on $Y$.  If $\max_X |\nabla v| \leq \max_Y |\nabla u|$, then we immediately conclude.  We may therefore choose a convex set $Z \subseteq X$ of the form $Z = X \cap \{ |\nabla v| > \tau \}$ such that $Z \cap Y = \emptyset$ and $\overline Z \cap Y \neq \emptyset$.

  Using the inner regularity of $V$ and the containment $V \subseteq U$, we see there are $\alpha, \delta > 0$ such that \lref{convextouch} applies to the function $g : \Lambda \to \R$ given by
  \begin{equation*}
    g(x) = \sup \{ t \in \R : x + t \nu \in U \mbox{ or } t \leq - \delta \}
  \end{equation*}
  and the set $Z \subseteq \Lambda$.  In particular, there is a radius $r > 0$, a sequence of points $y_n \in \partial U$, and a sequence of unit vectors $\nu_n \in \R^d$ such that $B_r(y_n + r \nu_n) \subseteq U \setminus \overline W$ touches $\partial U$ from the inside at $y_n$, $y_n \to y_\infty \in Y \cap \overline Z$, $\nu_n \to \nu$.  Call $y_n'$ to be the projection of $y_n$ onto $\Lambda$, then $y_n = y_n'+g(y_n)\nu$,  and $\nu_n \cdot y_n  \leq (1+o(1)) \min_{y \in \overline Z} \nu_n \cdot y$.  

  Comparing the first order expansions from \lref{C32} of $u$ at $y_n$ and $v$ at $y_\infty$, using the uniform bounds on the tangent ball radii, we see that
  \begin{equation*}
    \liminf_{n \to \infty} |\nabla u(y_n)| \geq \liminf_{n \to \infty} \nu \cdot \nabla u(y_n) \geq \nu \cdot \nabla v(y_\infty) = |\nabla v(y_\infty)| \geq \tau.
  \end{equation*}
  Now, for any sequence $z_n \in \overline V_{\nu_n}$ with $z_n \to z$, by the boundary regularity $z \in X$ and by the conditions $y_\infty \in Y \cap \overline Z \subseteq X \setminus Z$ and $\nu_n \cdot y_n \leq (1+o(1)) \min_{y \in \overline Z} \nu_n \cdot y$ we have that $z \in X \setminus Z$.  In particular, $\nu \cdot \nabla v(z) = |\nabla v(z)| \leq \tau$.  Thus, for any $\ep > 0$, for all large $n$, we have $|\nabla v| \leq (1 + \ep) |\nabla u(y_n)|$ on $\overline V_{\nu_n} = \overline V_{\nabla u(y_n)}$.
\end{proof}

The next lemma is used for the ``easy'' direction of our uniqueness theorem below.  This is essentially already present in the viscosity theory for $H(p) = |p|$, but needs to be adapted to work for weakened subsolutions.

\begin{lemma}
  \label{l.convexcomparison2}
  Suppose $W \subseteq V \subseteq U \subseteq \R^d$ are open and bounded, $W$ and $V$ are inner regular, $U$ is convex, $\overline W \subseteq V$, and $\partial V \cap \partial U \neq \emptyset$.  For any $\ep > 0$, we can find a point $x \in \partial U \cap \partial V$, a normal vector $\nu$, and a $\delta > 0$ such that $\nabla u_{W,U}(x) = |\nabla u_{W,U}| \nu$ exists and $w_{W,V}(y) \leq (1 + \ep) \nabla w_{W,U}(x) \cdot (y - x)$ for $\nu \cdot (y - x) < \delta$.
\end{lemma}

\begin{proof}
  Write $u = w_{W,U}$ and $v = w_{W,V}$.  Since $U$ is convex and $W$ is inner regular, $u$ is Lipschitz continuous.  Since $V$ is inner regular, we see from \lref{C32} that $\nabla u$ exists and is continuous in $U$ up to $\partial U \cap \partial V$.  Choose $x \in \partial U \cap \partial V$ and a unit vector $\nu$ such that $\nu \cdot x = \min_{y \in \overline U} \nu \cdot y$.  Let $\Lambda = \{ y \in \R^d : \nu \cdot (y - x) = 0 \}$ denote the supporting hyperplane of $\overline U$ at $x$ with normal $\nu$.  Consider the set $X = \Lambda \cap \partial V$.  Since $V \subseteq U$, we have $X \subseteq \partial U \cap \partial V$.  In particular, $\nabla u$ is continuous in $U$ up to $X$.  Since $\nabla u = |\nabla u| \nu$ on $X$, we may assume that $|\nabla u| \leq |\nabla u(x)|$ on $X$.  By \lref{C32above}, we may choose $r > 0$ such that $u(y) \leq (1 + \ep) \nabla u(x) \cdot (y - x)$ for all $y \in U \cap (X + B_r)$.  Using the definition of $X$, we may choose $\delta > 0$ such that $\{ y \in V : \nu \cdot (y - x) < \delta \} \subseteq X + B_r$.  Conclude using $v \leq u$.
\end{proof}

\subsection{Uniqueness}

Our uniqueness proof requires two additional and fairly standard lemmas.  The first says that, when a supersolution is differentiable at a boundary point, then the supersolution condition holds classically.

\begin{lemma}
  \label{l.regularsupersolution}
  If $u \in USC(\R^d)$ is a supersolution of \eref{exterior}, $\{ u > 0 \}$ has an inner tangent ball at $x \in \partial \{ u > 0 \}$, and there is a $p \in \R^d$ such that $u(y) = p \cdot (y - x) + o(|y - x|)$ for $y \in \{ u > 0 \}$, then $H(p) \leq 1$.
\end{lemma}

\begin{proof}
  Choose $B_r(z) \subseteq \{ u > 0 \} \setminus \overline W$ with $\partial B_r(z) \cap \partial \{ u > 0 \} = \{ x \}$.  Let $G$ denote the fundamental solution of $L$.  For $\ep > 0$ small, consider the test function
  \begin{equation*}
    \varphi_\ep(y) = (1 - \ep)|p| |\nabla G(x - z)|^{-1} (G(y-z) - G(x-z)).
  \end{equation*}
  Observe that, for $\delta > 0$ small, $\varphi_\ep$ is subharmonic and touches $u$ from below in $B_\delta(x)$ with contact set $\{ x \}$.  The supersolution condition implies $H(\nabla \varphi_\ep(x)) = H((1-\ep) p) \leq 1$.  Conclude using the homogeneity \eref{Hhomogeneous}.
\end{proof}

The next lemma is the usual min and max convolution.

\begin{lemma}
  \label{l.regularize}
  If $u \in USC(\R^d)$ is a (weakened) subsolution of \eref{exterior} for $W$, then $u^\delta = \max_{\overline B_\delta(x)}u$ is a (weakened) subsolution of \eref{exterior} for $W^\delta = W + B_\delta$.  If $u \in LSC(\R^d)$ is a supersolution of \eref{exterior} for $W$, then $u_\delta(x) = \min_{\overline B_\delta(x)} u$ is a supersolution of \eref{exterior} for $W_\delta = \R^d \setminus \overline{(\R^d \setminus \overline W)^\delta}.$
\end{lemma}

\begin{proof}
  This is a standard property of viscosity solutions.
\end{proof}

\noindent Note that the positivity sets for $u^\delta,u_\delta$ are respectively inner and outer regular.

Using our technical lemmas above, we are ready to prove uniqueness.

\begin{theorem}
  \label{t.uniqueness}
  Suppose $0 \in W \subseteq \R^d$ is open, bounded, convex, and inner regular.  There is a unique $u \in C_c(\R^d)$ that satisfies $u = 1$ on $\overline W$ and is both a supersolution and weakened subsolution of \eref{exterior}.  Moreover, $\{ u > 0 \}$ is convex.
\end{theorem}

\begin{proof}
  Step 1.
  Suppose that $u, v \in C_c(\R^d)$ are identically $1$ on $\overline W$ and both supersolutions and weakened subsolutions of \eref{exterior}.  Observe that $u = w_{W,U}$ and $v = w_{W,V}$ where $U = \{ u > 0 \}$ and $V = \{ v > 0 \}$.  To show $u = v$, it is enough, by symmetry, to show that $V \subseteq U$.  We assume for contradiction that $V \nsubseteq U$.  By \tref{existence} we may assume that one of $U$ or $V$ is convex.

  Step 2.
  By min and max convolution and dilation, we regularize $U$ and $V$ and make $u$ strict supersolution.  Since $V \nsubseteq U$ and $W$ is convex, we may choose $0 < \delta < 1 < \tau$ such that $V^\delta \subseteq \tau U_\delta$, $\partial V^\delta \cap \partial (\tau U_\delta) \neq \emptyset$, and $W^\delta \subseteq \tau W_\delta$.  Let $\tilde W = W^\delta$, $\tilde U = \tau U_\delta$, $\tilde V = V^\delta$, $\tilde v = w_{\tilde W, \tilde V}$, and $\tilde u = w_{\tilde W, \tilde U}$.  Using \lref{regularize} and \lref{harmonize}, we see that $\tau \tilde u$ is a supersolution and $\tilde v$ is a weakened subsolution.  We also have that $\tilde U$ is outer regular, $\tilde V$ is inner regular, $\tilde W$ is convex and inner regular, $\tilde V \subseteq \tilde U$, and $\tilde V \cap \tilde U \neq \emptyset$.

  Step 3.
  We derive a contradiction assuming that $U$ is convex.  Since $\tilde U$ is convex, we may apply \lref{convexcomparison2}.  For any $\ep > 0$, we find a $\delta > 0$, a normal $\nu$, and a point $x \in \partial \tilde U \cap \partial \tilde V$ such that $\nabla \tilde u(x) = |\nabla \tilde u(x)| \nu$ exists and $\tilde v(y) \leq (1+\ep) \cdot \nabla \tilde u(x)$ for $\nu \cdot (y - x) < \delta$.  By \lref{regularsupersolution}, we have $H(\tau \nabla \tilde u(x)) \leq 1$.  In particular $H((1+\ep) \nabla \tilde u(x)) \leq (1+\ep) \tau^{-1}$.  Since $\tau > 1$, we can choose $\ep > 0$ small so that $(1 + \ep) \tau^{-1} < 1$.  Thus $\tilde v$ fails the weak subsolution condition for the test function $\varphi(y) = (1+\ep) \nabla \tilde v(x) \cdot (y - x)$.

  Step 4.
  We derive a contradiction assuming that $V$ is convex.  Since $\tilde V$ is convex and inner regular, we may apply \lref{convexcomparison1}.  For any $\ep > 0$, we find a $x \in \partial \tilde V$ where $p = \nabla \tilde v(x) \neq 0$ exists and $|\nabla \tilde u(x)| \leq (1 + \ep) |p|$ on $\overline {\tilde V}_p$.  Since $\tilde V$ is convex and inner regular, we see from \lref{C32} that $\nabla \tilde v$ is continuous up to $\partial \tilde V$.  In particular, we see there is a $\delta > 0$ such that $\tilde v(y) \leq (1 + 2 \ep) p \cdot (y - x)$ holds for any $x \in \overline {\tilde V}_p$ and $p \cdot (y - x) < \delta$.  As in step 3, \lref{regularsupersolution} implies that $H((1+2\ep)p) \leq (1 + 2\ep)\tau^{-1}$.  Thus, making $\ep > 0$ small, we see that $\tilde v$ fails the weak subsolution condition for the test function $\varphi(y) = (1 + 2 \ep) p \cdot (y - x)$.
\end{proof}

\subsection{Rational facets}

We now use the special structure of $H$ given by \tref{ridges} to show that facets form in all of the rational directions.  Observe that, when $p \in \Z^d \setminus \{ 0 \}$ is rational, then the lattice $\Lambda_p$ determines $|p|^{-1} p$.  Then by \tref{ridges} there is $\delta > 0$ such that $H(|p|^{-1} p) \leq H(q) - \delta$ for all $q \in \R^d$ with $|q| = 1$ and $0 < |q - |p|^{-1} p| < \delta$.  Combining this with a compactness argument inspired by free boundary regularity theory, we show that a facet forms in the direction $p$.

We remark that the existence of facets, even of higher co-dimension, is easier if $|Du|$ is continuous up to $\partial \{ u>0\}$.  This would follow from a $C^{1,\textup{Dini}}$ estimate of the free boundary.  Here we use convexity to get $C^1$ regularity of the free boundary cheaply, and use a blow up argument in place of regularity of the gradient.  

\begin{theorem}
  \label{t.facets}
  Suppose that $W \subseteq \R^d$ is open, bounded, convex, and inner regular, let $u \in C_c(\R^d)$ denote the unique solution of \eref{exterior}, and let $\Omega = \{ u > 0 \}$.  For every $p \in \Z^d \setminus \{ 0 \}$, the set $\overline \Omega_p = \{ x \in \overline \Omega : p \cdot x = \min_{y \in \overline \Omega} p \cdot y \}$ is closed, convex, and has dimension $d-1$.
\end{theorem}

\begin{proof}
Recall from \tref{uniqueness} that $\Omega$ is convex, first we prove that the boundary of $\Omega$ is $C^1$.  Each $x \in \partial \Omega$ has a closed convex cone of supporting half-spaces indexed by their inward normals, we write $\nu(x)$ for the intersection of the cone with the unit sphere.  First we show that $\nu$ is single-valued.  Let $x \in \partial \Omega$ and let $e$ be a unit vector such that $e \cdot \nu >0$ for all $\nu \in \nu(x)$.   Blow up to 
\[ u_r(y) = u(re)^{-1}u(x + ry).\]
Up to taking a subsequence, by the uniform Lipschitz continuity of $u$ and the Harnack inequality, the $u_r$ converge locally uniformly to a positive globally Lipschitz continuous harmonic function $u_0$ in the cone $K = \{y \cdot \nu >0 \ \forall \ \nu \in \nu(x)\}$, with $u_0(e) = 1$ and $u_0 = 0$ on the boundary of the cone.  Thus, for example by Theorem 1.2 of \cite{Armstrong-Sirakov-Smart}, $u(x) = (x \cdot f)_+$ for some $f \neq 0$, i.e. $\nu(x) = \{ f/ |f|\}$ is a singleton.  Now since $\nu: \partial \Omega \to S^{d-1}$ is single-valued, it is also continuous.  Suppose otherwise, then there exists a point $x \in \partial \Omega$ and a sequence $x_n$ converging to $x$ with $\nu(x_n) \to \nu' \neq \nu$.  Then, by the continuity of $p \mapsto \min_{\overline{\Omega}} y\cdot p$,  $\{\nu'\cdot x =0\}$ is also a supporting hyperplane for $\Omega$ at $x$, this is a contradiction.

  Since $\Omega$ is convex, $\overline \Omega_p$ is closed and convex.  Write $\bar p = |p|^{-1} p$.  Since $p$ has integer coordinates, \tref{ridges} implies there is a small $s > 0$ such that $H(q) \geq (1 + 2s) H(\bar p)$ whenever $|q| = 1$ and $0 < |q - \bar p| < s$.  Suppose for contradiction that $\overline \Omega_p$ has dimension $d-2$.

  By the weak subsolution condition, there is a sequence of points $x_n \in \Omega$ with $x_n \to x_\infty \in \overline \Omega_p$ such that
  \begin{equation*}
    \limsup_{n \to \infty} \frac{u(x_n)}{\bar p \cdot (x - x_n)} \geq H(\bar p)^{-1}.
  \end{equation*}
  By convexity, we have
  \begin{equation*}
    \bar p \cdot (x_n - x_\infty) \geq \dist(x_n, \partial \Omega).
  \end{equation*}
  Since $u \in C^{0,1}(\R^d)$, we have
  \begin{equation*}
    \bar p \cdot (x_n - x_\infty) \leq C u(x_n) \leq C \dist(x_n, \partial \Omega).
  \end{equation*}
  
  We now use the low dimension hypothesis.  Recall that, for any bounded convex open $U \subseteq \R^d$ and sphere $\partial B \subseteq U$, the set
  \begin{equation*}
    V = \{ y \in \partial U : \dist(z,\partial U) = |z - y| \mbox{ for some } z \in \partial B \}
  \end{equation*}
  has positive $\mathcal H^{d-1}$ measure.  Indeed, the downward characteristics of the distance function in a convex set spread out as they travel to the boundary.  From this, we see that, by perturbing the $x_n$, we may assume there are $y_n \in \partial \Omega$ such that $r_n = |y_n - x_n| = \dist(x_n, \partial \Omega)$ and $q_n = |x_n - y_n|^{-1} (x_n - y_n) \neq \bar p$.  Since $y_n \to x_\infty$ and we have shown above that $\partial \Omega$ is $C^1$ we have $q_n \to \bar p$.

  By \lref{C32} and \lref{regularsupersolution}, $\nabla u(y_n)$ exists, $\nabla u(y_n) = |\nabla u(y_n)| q_n$, and $|\nabla u(y_n)| \leq H(q_n)^{-1}$.  Moreover, we may choose a $\delta \in (0,1)$ such that
  \begin{equation*}
    z_n = (1 - \delta) y_n + \delta x_n
  \end{equation*}
  satisfies
  \begin{equation*}
    u(z_n) \leq (1+s) H(q_n)^{-1} \dist(z_n, \partial \Omega).
  \end{equation*}

  Consider the rescalings
  \begin{equation*}
    u_n(w) = u(x_n)^{-1} u(y_n + r_n w).
  \end{equation*}
  Since $C^{-1} r_n \leq u(x_n) \leq C r_n$, we may pass to a subsequence to obtain $u_n \to u \in C^{0,1}(\R^d)$ locally uniformly, $r_n^{-1} (x_n - y_n) \to x$, and $r_n^{-1} (z_n - y_n) \to \delta x$.  Note that $u(0) = 0$, $u(x) = 1$, and $B_1(x) \subseteq \{ u > 0 \}$.  In fact, by the same argument as in \tref{existence}, we see that $\{ u > 0 \}$ is convex and unbounded.

  Since $\{ u > 0 \}$ is convex and unbounded, $u$ is globally Lipschitz, harmonic, $u(0) = 0$, $u(x) = 1$, and $|x| = 1$, we must have $u(w) = \max \{ 0, x \cdot w \}$.  Thus compute
  \begin{equation*}
    \begin{aligned}
      \delta & = u(\delta x) \\
      & \leq \lim_n u(x_n)^{-1} u(z_n) \\
      & \leq \limsup_n H(p) (1+s)H(q_n)^{-1} \delta \\
      & \leq (1+s) (1+2s)^{-1} \delta \\
      & < \delta,
    \end{aligned}
  \end{equation*}
  which is impossible.  Thus $\overline \Omega_p$ must have dimension $d-1$.
\end{proof}

\section{Scaling Limit}
\label{s.scaling}

\subsection{Local minimizers}

Throughout this section, fix $0 \in W \subseteq \R^d$ open, bounded, convex, and inner regular.  For $h > 1$, consider the energy
\begin{equation*}
  J_h[u] = \sum_{x \notin h W} \id_{\{ u > 0 \}}(x) + d \sum_{\substack{|x-y|=1 \\ \{ x, y \} \nsubseteq h W}} (u(x)-u(y))^2
\end{equation*}
defined for $u : \Z^d \to \R$.  Note that the sum over edges counts each edge twice, as the pairs $\{ x, y \}$ are unordered.

We define local minimizer to be with respect to single-site modifications:

\begin{definition}
  A function $u : \Z^d \to \R$ is a local minimizer of $J_h$ if $u = h$ on $h W$, $J_h[u] < \infty$, and $J_h[v] \geq J_h[u]$ for any $v : \Z^d \to \R$ such that $\{ v \neq u \}$ is a singleton.
\end{definition}

An analysis of single-site modifications yields the following result.

\begin{lemma}
  If $u : \Z^d \to \R$ is a local minimizer of $J_h$, then
  \begin{equation}
    \label{e.firstvariation}
    \begin{cases}
      \Delta u = 0 & \mbox{on } \{ u > 0 \} \setminus h W \\
      \Delta u \leq 1 & \mbox{on } \partial^+ \{ u > 0 \} \setminus h W \\
      u \geq (2d)^{-1} & \mbox{on } \partial^- \{ u > 0 \} \setminus h W \\
      u = h & \mbox{on } h W \\
      u \geq 0 & \mbox{on } \Z^d \setminus h W,
    \end{cases}
  \end{equation}
  where
  \begin{equation*}
    \partial^+ X = \{ y \in \Z^d \setminus X : |y - x| = 1 \mbox{ for some } x \in X \}
  \end{equation*}
  and
    \begin{equation*}
    \partial^- X = \partial^+ (\Z^d \setminus X)
  \end{equation*}
  denote the outer and inner lattice boundary of a set $X \subseteq \Z^d$.
\end{lemma}

\begin{proof}
  Since the first, fourth, and fifth conditions are standard, we prove the second condition.  Suppose that $x \in \partial^+ \{ u > 0 \} \setminus h W$ and consider the variation
  \begin{equation*}
    u^x(y) = \begin{cases}
      u(y) & \mbox{if } y \neq x \\
      (2d)^{-1} \sum_{|z - x| = 1} u(z) & \mbox{if } y = x.
    \end{cases}
  \end{equation*}
  Using $u(x) = 0$, we compute
  \begin{equation*}
    0 \leq J_h[u^x] - J_h[u] = 1 - (\Delta u(x))^2.
  \end{equation*}
  This proves the second condition.  Now suppose that $x \in \partial^- \{ u > 0 \} \setminus h W$ and consider the variation
  \begin{equation*}
    u_x(y) = \begin{cases}
      u(y) & \mbox{if } y \neq x \\
      0 & \mbox{if } y = x.
    \end{cases}
  \end{equation*}
  Using $\Delta u(x) = 0$, we compute
  \begin{equation*}
    0 \leq J_h[u_x] - J_h[u] = 1 - (2d)^2 u(x)^2.
  \end{equation*}
  This proves the third condition.
\end{proof}

We study the least supersolution of \eref{firstvariation}.

\begin{lemma}
  The least function $u_h : \Z^d \to \R$ that satisfies
  \begin{equation}
    \label{e.fde}
    u_h \geq h \id_{h W} \quad \mbox{and} \quad \Delta u_h \leq \id_{\{ u_h = 0 \}}
  \end{equation}
  is a local minimizer.
\end{lemma}

\begin{proof}
  By the maximum principle for $\Delta$, we see that $u_h = h$ on $h W$, $\{ u_h > 0 \}$ is finite, and $\Delta u_h = 0$ in $\{ u_h > 0 \} \setminus h W$.  If $J_h[u] < J_h[u_h]$ and $\{ u \neq u_h \} = \{ x \}$, then either $x \in \partial^+ \{ u_h > 0 \} \setminus hW$ or $x \in \partial^- \{ u_h > 0 \}$.  In the former case, then by the proof of the previous lemma, we must have $\Delta u_h(x) > 1$, which is impossible.  In particular, any single site variation of $u_h$ that reduces the energy must remove a point from its domain.  We may therefore find a local minimizer $v : \Z^d \to \R$ such that $\{ v > 0 \} \subsetneq \{ u > 0 \}$.  By the previous lemma, $\Delta v \leq \id_{\{ v > 0 \}}$, contradicting the minimality of $u$.
\end{proof}

As in the continuum setting, we may compute the least supersolution by parabolic flow.  This flow is a special case of the divisible sandpile dynamics.

\begin{lemma}
  \label{l.flow}
  If $v_0 = h \id_{h W}$,
  \begin{equation}
    \label{e.flow}
    v_{k+1} = v_k + (2d)^{-1} \max \{ 0, \Delta v_k - \id_{\{ v_k = 0 \}} \},
  \end{equation}
  and $u_h : \Z^d \to \R$ the least function satisfying \eref{fde}, then $\lim_{k \to \infty} v_k = u_h$.
\end{lemma}

\begin{proof}
  We first observe that $0 \leq v_k \leq h$ for all $k$.  Next, we prove $v_k \leq u_h$ by induction.  It is trivially true for $k = 0$.  Suppose it is true for $k$ and let $x \in \Z^d$ be arbitrary.  If $x \in hW$, then $v_{k+1}(x) \leq h = u_h(x)$.  If $u_h(x) = 0$, then we have $v_k(x) = 0$, $\Delta v_k(x) \leq \Delta u_h(x) \leq 1$, and $v_{k+1}(x) = 0$.  If $u_h(x) > 0$ and $x \notin h W$, then we have $v_{k+1}(x) \leq (2d)^{-1} \sum_{|y-x|=1} v_k(y) \leq (2d)^{-1} \sum_{|y-x|=1} u_h(y) = u_h(x)$.

  Since $v_k \leq v_{k+1} \leq u_h$ and $\{ u_h > 0 \}$ is finite, we see that $v_\infty = \lim_{k \to \infty} v_k$ exists.  Since $v_\infty$ is stationary, we see that $v_\infty$ satisfies \eref{fde}.  We conclude $v_\infty = u$ using the minimality of $u_h$.
\end{proof}

\subsection{Compactness}

We adapt the uniform Lipschitz estimate of Alt-Caffarelli \cite{Alt-Caffarelli} to the discrete setting.  A variation on this already appeared in the work of Aleksanyan-Shahgholian \cite{Aleksanyan-Shahgholian,Aleksanyan-Shahgholian-2}.  In fact, the discrete analogue of the Alt-Caffarelli result is easier to prove, owing to the fact that a discrete boundary can not be irregular and the discrete normal derivative is always defined.

\begin{lemma}
  \label{l.uhlipschitz}
  Every local minimizer of $J_h$ is Lipschitz with constant depending only on dimension and the inner regularity of $W$.
\end{lemma}

\begin{proof}
  We show that, if $u : \Z^d \to \R$ is a local minimizer and $x \in \partial^+ hW$, then $u(x) \geq h - C$ for some $C > 0$ depending only on dimension and the inner regular of $W$.  Let us first observe that this is sufficient.  Indeed, by \eref{firstvariation}, we know that $u(x) \leq C$ for all $x \in \partial^+ \{ u = 0 \}$.  Thus, since $\Delta u = 0$ in $\{ u > 0 \} \setminus h W$, we can apply the maximum principle to $x \mapsto u(x+e_k) - u(x)$ to conclude $|u(y)-u(x)| \leq C$ for $|x - y| = 1$.

  Suppose that $W$ is $2r$-inner regular and $r < 1/2$.  Define the test function
  \begin{equation*}
    \varphi(x) = 2 r^{d-1} |x|^{1-d} - 1
  \end{equation*}
  and its rescalings
  \begin{equation*}
    \varphi_h(x) = h \varphi(h^{-1} x).
  \end{equation*}
  Observe that $\varphi_h = h$ on $\partial B_{hr}$ and $\varphi_h = 0$ on $\partial B_{h s}$, where $s = r 2^{1/(d-1)} > r$.  Since $L \varphi(x) = d r^{d-1} |x|^{-1-d} \geq C^{-1} > 0$ in $B_s \setminus B_r$, we see that
  \begin{equation*}
    \Delta \varphi_h > 0 \quad \mbox{on } B_{hs} \setminus \overline B_{hr}
  \end{equation*}
  for large $h > 1$.  Using $|\nabla \varphi|(x) = (d-1) r^{d-1} |x|^{-d}$ and $r < 1$ and $r < 1/2$, we also see that
  \begin{equation*}
    \Delta \max \{ 0, \varphi_h \} > 1 \quad \mbox{on } \partial B_{hs}
  \end{equation*}
  for large $h > 1$.
  
  We claim, for $h > 1$ large, that
  \begin{equation*}
    u(w) \geq \varphi_h(w - z) \quad \mbox{for } w \in \Z^d.
  \end{equation*}
  Indeed, if this fails, then by the maximum principle, the maximum of $w \mapsto \varphi_h(w - z) - u(w)$ must occur at a point $w \in \partial^+ \{ u > 0 \} \setminus h W$.  However, this implies $\Delta u(w) \geq \Delta \max \{ 0, \varphi_h \}(w) > 1$, contradicting \eref{firstvariation}.

  Since, for large $h > 1$, we have $hs \geq hr + 2 \geq |x- z| \geq hr$, the claim implies $u(x) \geq \varphi_h(w-z) \geq h - C$.  For small $h > 1$, we have $u(x) \geq 0 \geq h - C$ trivially.
\end{proof}

We also must control the support to obtain compactness.

\begin{lemma}
  \label{l.uhsupport}
  There is a radius $R > 0$ depending only on $W$ such that every local minimizer of $J_h$ is supported in $B_{Rh}$.
\end{lemma}

\begin{proof}
  Let $u : \Z^d \to \R$ be a local minimizer of $J_h$.  Choose $r > 0$ such that $W \subseteq B_r$.  For $R > r$, let $v : \Z^d \to \R$ solve
  \begin{equation*}
    \begin{cases}
      \Delta v = 0 & \mbox{on } B_{hR} \setminus B_{hr} \\
      v = 0 & \mbox{on } \Z^d \setminus B_{hR} \\
      v = h & \mbox{on } B_{hr}.
    \end{cases}
  \end{equation*}
  Standard Green's function estimates tell us that $\Delta v \leq C (R - r)^{1-d} \id_{\{ v = 0 \}}$.  In particular, $\Delta v \leq (4 d)^{-1} \id_{\{ v = 0 \}}$ for all large $R > 0$.  Now, let $R > r$ be least such that $B_{hR} \supseteq \{ u > 0 \}$.  Choose $x \in \partial^+ \{ u > 0 \} \setminus B_{hR}$.  By \eref{firstvariation} and the maximum principle, we have $(2d)^{-1} \leq \Delta u(x) \leq \Delta v(x)$.  Thus $\{ u > 0 \} \subseteq B_{hR}$ for some $R > r$ depending only on dimension and $W$.
\end{proof}

\subsection{The supersolution property}

The supersolution property for any limit of local minimizers follows by a blow-up argument.

\begin{lemma}
  If $p \in \R^d$ and, for every $R, \delta > 0$ there is a $v : \Z^d \to \R$ such that
  \begin{equation*}
    \begin{cases}
      v(0) = 0 \\
      \Delta v(x) \leq \delta + \id_{\{ v = 0 \}}(x) & \mbox{for } |x| \leq R \\
      v(x) \geq (p \cdot x - \delta)^+ & \mbox{for } |x| \leq R
    \end{cases}
  \end{equation*}
  then $H(p) \leq 1$.
  
\end{lemma}

\begin{proof}
  By the maximum principle, we may assume $v$ satisfies the additional conditions
  \begin{equation*}
    v(x) \leq (p \cdot x + \max_k |p_k|)^+ \quad \mbox{and} \quad \Delta v(x) \geq 0 \quad \mbox{for } |x| \leq R.
  \end{equation*}
  More precisely take the min with $(p \cdot x + \max_k |p_k|)^+$ and then solve for the harmonic function in $\{ |x| \leq R\} \cap \{ p \cdot x > - \delta\}$ with the same boundary data. Together with the original conditions, these imply enough compactness to pass to limits, obtaining $w : \Z^d \to \R$ that solves \eref{halfspace} and has $\Delta w(0) \leq 1$.
\end{proof}

\begin{lemma}
  \label{l.uhsupersolution}
  If $h_n \to \infty$, $u_n$ is a local minimizer of $J_h$, $\bar u_n(x) = h_n^{-1} u_n(h_n x)$, and $\bar u_n \to u \in C(\R^d)$ uniformly, then $u$ is a supersolution of \eref{exterior}.
\end{lemma}

\begin{proof}
  Suppose that $\varphi \in C^\infty(\R^d)$ touches $u$ from below at $x$.  That is $u(x) = \varphi(x)$ and, for some $r > 0$ and all $0 < |y - x| < r$, $u(y) > \varphi(y)$.  It is enough to show $H(D \varphi(x)) \leq 1$ under the assumption $\Delta \varphi > 0$ in $B_r(x) \subseteq \R^d \setminus B_r$.  Let
  \begin{equation*}
    U_n = \Z^d \cap B_{h_n r}(h_n x) \quad \mbox{and} \quad \varphi_n(x) = h_n \varphi(h_n^{-1} x).
  \end{equation*}
  Choose $x_n \in U_n$ such that
  \begin{equation*}
    (u_n - \varphi_n)(x_n) = \min_{U_n} (u_n - \varphi_n).
  \end{equation*}
  By the uniform convergence of $h_n^{-1} u_n(h_n x)$ to $u(x)$ in $B_r(x)$, we see that
  \begin{equation*}
    \frac{x_n}{h_n} \to x.
  \end{equation*}
  By the monotonicity of $\Delta$, we have
  \begin{equation*}
    \Delta (u_n - \varphi_n)(x_n) \geq 0.
  \end{equation*}
  Since
  \begin{equation*}
    \Delta (u_n - \varphi_n) < \id_{\{ u_n = 0 \}},
  \end{equation*}
  we must have
  \begin{equation*}
    u_n(x_n) = 0.
  \end{equation*}
  Letting
  \begin{equation*}
    v_n(x) = u_n(x_n + x)
  \end{equation*}
  we see that
  \begin{equation*}
    \begin{cases}
      v_n(0) = 0 \\
      \Delta v_n \leq \id_{\{ v_n = 0 \}} \\
      v_n \geq 0 \\
      v_n(x) \geq D \varphi_n(x_n) \cdot x + O( h_n^{-2} |x|^2) & \mbox{for } |x| \leq h_n r.
    \end{cases}
  \end{equation*}
  Since $D\varphi_n(x_n) \to D \varphi(x)$, the previous lemma implies $H(D\varphi(x)) \leq 1$.
\end{proof}

\subsection{The weakened subsolution property}

The weakened subsolution property for any limit of least supersolutions also follows by a blow-up argument.  Note that this property generically fails for arbitrary sequences of local minimizers.

\begin{lemma}
  \label{l.uhweakenedsubsolution}
  If $h_n \to \infty$, $u_n : \Z^d \to \R$ is the least function satisfying \eref{fde}, $\bar u_n(x) = h_n^{-1} u_n(h_n x)$, and $\bar u_n \to u \in C(\R^d)$ uniformly, then $u$ is a weakened subsolution of \eref{exterior}.
\end{lemma}

\begin{proof}
  Since $\Delta u_n = 0$ in $\{ u_n > 0 \} \setminus h_n W$, we see that $L u = 0$ in $\{ u > 0 \} \setminus W$.  In particular, we may suppose for contradiction that $H(p) < 1$, $U \subseteq \R^d \setminus \overline W$ open, and $\varphi(y) = p \cdot (y - x)$ touches $u$ from above in $\overline{\{ u > 0 \}} \cap U$.

  We may assume that $\overline U$ is compact.  By the strong maximum principle, we see that the contact set is a compact subset of $\partial \{ u > 0 \} \cap U$.  We may therefore choose a $\delta > 0$ such that
  \begin{equation*}
    \{ \varphi - \delta < u \} \cap \overline{ \{ u > 0 \}} \cap U
  \end{equation*}
  is non-empty and has compact closure in $\overline{\{ u > 0 \}} \cap U$.

  Let $v$ denote the solution of \eref{halfspace} for $p$.  We may select a sequence of points $x_n$ such that, if $v_n(x) = v(x - x_n)$, then $\bar v_n(x) = h_n^{-1} v_n(h_n x)$ converges locally uniformly to $\psi(x) = \max \{ 0, \varphi(x) - \delta \}$.  By the above, and the uniform convergence of $\bar u_n$ to $u$, we see that, for all sufficiently large $n$, we have
  \begin{equation*}
    \emptyset \neq \{ v_n < u_n \} \cap \{ u_n > 0 \} \cap h_n U \subseteq (h_n U_n) \setminus \partial^-(h_n U).
  \end{equation*}

  In particular, if we define
  \begin{equation*}
    \tilde u_n(x) = \begin{cases}
      \min \{ u_n(x), v_n(x) \} & \mbox{if } x \in h_n U \\
      u_n(x) & \mbox{otherwise},
    \end{cases}
  \end{equation*}
  then $\tilde u_n$ is a strictly smaller solution of \eref{fde}.
\end{proof}

\subsection{Convergence}

Appealing to our uniqueness theorem for viscosity solutions, we see that the minimal supersolutions converge.

\begin{theorem}
  \label{t.convergence}
  If $W \subseteq \R^d$ is open, bounded, convex, and inner regular and $u_h : \Z^d \to \R$ is the least function satisfying
  \begin{equation*}
    u_h \geq h \id_{h W} \quad \mbox{and} \quad \Delta u_h \leq \id_{\{u_h = 0 \}},
  \end{equation*}
  then the rescalings $\bar u_h(x) = h^{-1} u_h(h x)$ converge uniformly to the unique solution of \eref{exterior}.
\end{theorem}

\begin{proof}
  By \lref{uhlipschitz} and \lref{uhsupport}, every subsequence of $h_n \to \infty$ has a further subsequence $h'_n \to \infty$ such that $\bar u_{h'_n} \to u \in C^{0,1}_c(\R^d)$ uniformly.  By \lref{uhsupersolution} and \lref{uhweakenedsubsolution}, $u$ is a supersolution and weakened subsolution of \eref{exterior}.  Thus, by \tref{uniqueness}, $u$ is the unique solution of \eref{exterior}.
\end{proof}

\section{The non-convex case}
\label{s.nonconvex}

\subsection{Failures of uniqueness}

As is well-known, see Caffarelli-Salsa \cite{Caffarelli-Salsa}, the exterior problem \eref{exterior} does not always have a unique viscosity solution even for the simplest case $H(p) = |p|$.  Indeed, when $W$ is the complement of a ball or $W$ is a finite union of disjoint balls, there generally are multiple viscosity solutions.  

In the situations where uniqueness is known to hold for $H(p) = |p|$, the argument always consists of two steps.  First, one shows that supersolutions perturb to strict supersolutions.  Second, one applies strict comparison.  Since the known strictification arguments still work for general $H$, we are only missing strict comparison.  In this section, we describe some partial progress.

\subsection{Strict comparison in the plane}

We are not able to prove strict comparison for weakened subsolutions.  Instead, we need to use something closer to the full subsolution condition, which we now describe.  In order to state strict comparison in greater generality, we also make our definitions local.  That is, we define what it means to solve
\begin{equation}
  \label{e.localpde}
  \begin{cases}
    L u = 0 & \mbox{in } \{ u > 0 \} \\
    H(\nabla u) = 1 & \mbox{on } \partial \{ u > 0 \}
  \end{cases}
\end{equation}
in an open set $U \subseteq \R^d$, without any boundary conditions.

For convenience, we use a notion of local supersolution that is stronger than the natural local version of our notion of supersolution for \eref{exterior} above.   Indeed, we demand that local supersolutions be harmonic, rather than merely superharmonic, where they are positive.  Since we are no longer proving existence, this definition is more convenient.

\begin{definition}
  A supersolution of \eref{localpde} in an open set $U \subseteq \R^d$ is a non-negative $u \in C(U)$ that harmonic in $\{ u > 0 \}$ and such that, whenever $\varphi \in C^\infty(U)$ touches $u$ from below in $U$, there is a contact point $x$ where either $L \varphi(x) \leq 0$ or $\varphi(x) = 0$ and $H(\nabla \varphi(x)) \leq 1$.
\end{definition}

Our notion of subsolution is a natural interpolation of subsolution and weakened subsolution.  The idea is to exploit the linear structure of the discontinuities of $H$ as described in \tref{ridges}.

\begin{definition}
  A modified subsolution of \eref{localpde} in an open set $U \subseteq \R^d$ is a non-negative $u \in C(U)$ that is harmonic in $\{ u > 0 \}$ and such that, whenever $\varphi \in C^\infty(U)$, $\Sigma \subseteq \R^d$ a subspace, $D \varphi \subseteq \Sigma$, and $\varphi$ touches $u$ from above in $U \cap \overline{\{u > 0\}}$, there is a contact point $x$ where either $$L \varphi(x) \geq 0$$ or $$\varphi(x) = 0 \quad \mbox{and} \quad \limsup_{\substack{p \to \nabla \varphi(x) \\ p \in \Sigma}} H(p) \geq 1.$$
\end{definition}

We now prove strict comparison in dimension two.  The proof does use some additional structure of $H$ besides lower semi-continuity, which is that 
\[\lim_{q \to p, \ q\neq p} H(q) = H_*(p)\]
 exists. The idea is that if a smooth strict supersolution $u$ touches a smooth modified subsolution $v$ from above at a free boundary point, the gradient $p$ will have to satisfy $H_*(p) \geq 1$.  The interesting case is when $H$ is discontinuous at $p$.  Then $H_*(p) \geq 1$ contradicts the strict supersolution condition unless $Du$ is parallel to $p$ in a neighborhood of the touching point.  This argument can be repeated until the free boundary of $v$ separates from the free boundary of $u$, and we find a neighborhood of the contact set where $\partial \{ u>0\}$ is flat.  Then the weakened subsolution condition can be applied to get a contradiction.  The main additional difficulty, beyond this idea, is to deal with the reduced regularity of general sub/supersolutions.

\begin{theorem}
  \label{t.planestrictcomparison}
  If $U \subseteq \R^2$ is open, $u \in C(U)$ is a supersolution of \eref{localpde}, $v \in C(U)$ is a modified subsolution of \eref{localpde}, $\{ u > 0 \}$ is outer regular, $\{ v > 0 \}$ is inner regular, and $\tau \in (0,1)$, then $\tau u$ does not touch $v$ from above in $U \cap \overline{\{ v > 0 \}}$.
\end{theorem}

\begin{proof}
  Step 1.
  We suppose for contradiction that $\tau u$ touches $v$ from above in $U \cap \overline{\{ v > 0 \}}$.  Let $x$ be a contact point.  Observe that $\{ u > 0 \} \supseteq \{ v > 0 \}$.  Moreover, by the strong maximum principle, we must have $x \in \partial \{ u > 0 \} \cap \partial \{ v > 0 \}$.  By scaling, we may assume that $\{ v > 0 \}$ is $2$-inner regular and $\{ u > 0 \}$ is $2$-outer regular.  We may therefore choose a unit vector $\nu$ such that $$B_2(x + 2 \nu) \subseteq \{ v > 0 \} \subseteq \{ u > 0 \} \subseteq \R^d \setminus B_2(x - 2 \nu).$$  Using the supersolution property, we obtain
   $$\| u \|_{C^{0,1}(U)} \leq C.$$  Applying \lref{C32}, we see there is a $p \in \R^d$ such that
  \begin{equation*}
    |u(y) - p \cdot (y - x)| \leq C |y - x|^{4/3} \quad \mbox{for } y \in \{ u > 0 \}.
  \end{equation*}
  Note that $p = |p| \nu$ and $H(p) \leq 1$.

  Step 2.
  We show there is a $\sigma > 0$ such that $$H(q) \geq 1 + \sigma \quad \mbox{when } |q| = 1 \mbox{ and } 0 < |q - p| < \sigma.$$  Consider the test function
  \begin{equation*}
    \varphi(y) = |p| |\nabla G(\nu)|^{-1} (G(\nu) - G(y + \nu)),
  \end{equation*}
  where $G$ is the fundamental solution of $L$.  Observe that, for $\ep > 0$ small, there is a $\delta > 0$ such that $\psi_\ep = (1 + \ep) \varphi - \ep^2 \varphi^2$ is strictly superharmonic in $B_\delta(x)$ and touches $u$ from above in $B_\delta(x) \cap \overline{\{ u > 0 \}}$ with contact set $\{ x \}$.  In particular, $\tau \psi_\ep$ touches $v$ from above in $B_\delta(x) \cap \overline{\{ v > 0 \}}$ with contact set $\{ x \}$.  Recall from \tref{spikes} that there is a $k > 0$ such that $\lim_{q \to p} H(q) = k |p|$.  Since $L \psi_\ep(x) < 0$, the modified subsolution property implies that $\tau k |\nabla \psi_\ep(x)| \geq 1$.  Sending $\ep \to 0$ gives $\tau k |p| \geq 1$.  Since $H(p) \leq 1$, we must have $H(p) \leq \tau k |p| < k |p| = \lim_{q \to p} H(q)$.

  Step 3.
  We use this to show that $\partial \{ u > 0 \}$ contains an open line segment that contains $x$.  Consider the test function
  \begin{equation*}
    \varphi(y) = |p| |\nabla G(-\nu)|^{-1} (G(y - \nu) - G(-\nu)).
  \end{equation*}
  For $\ep > 0$ small, there is a $\delta > 0$ such that $\psi_\ep = (1 - \ep) \varphi + \ep^2 \varphi$ is subharmonic in $B_\delta(x)$ and touches $u$ from below in $B_\delta(x)$ with contact set $\{ x \}$.

  By the previous step, we fix $\ep > \delta > 0$ so that
  \begin{equation*}
    H(\nabla \psi_\ep(z)) > 1 \quad \mbox{for } z \in \{ \psi_\ep = 0 \} \cap B_\delta(x) \setminus \{ x \}.
  \end{equation*}
  
  Let us fix $\ep > \delta > 0$ with $(1 + \ep) \tau < 1$ and, for $|\xi| = 1$ and $0 < |\xi - \nu| < \delta$, $H(\xi) \geq (1 + \ep) H(\nu)$.  Let $\Lambda = \{ y \in \R^d : \nu \cdot (y-x) = 0 \}$.  Using $B_2(x + 2 \nu) \subseteq \{ u > 0 \}$, we have
  \begin{equation*}
    \inf_{y \in \{ u > 0 \} \cap \partial B_\delta(x)} u(y) - \psi_\ep(y) > 0.
  \end{equation*}
  In particular, there are $\gamma \in (0,\delta/2)$, $\alpha > 0$, and $h : B_\gamma(x) \cap \Lambda \to \R$ such that $|h(z)| \leq \alpha |z - x|^2$ and 
  \[\psi_{\ep,z}(y) = \psi_\ep(y - z - \nu h(z))\]
   touches $u$ from below in $B_{\delta/2}(x)$.  Since $H(\nabla \psi_\ep(z)) > 1$ when $z \in \{ \psi_\ep = 0 \} \cap B_\delta(x) \setminus \{ x \}$, the supersolution property implies that the contact set between $u$ and $\psi_{\ep,z}$ must be $\{ z + \nu h(z) \}$.  In particular, we see that
  \begin{equation*}
    \{ z + \nu h(z) : z \in \Lambda \cap B_\gamma(x) \} \subseteq \partial \{ u > 0 \}.
  \end{equation*}
  Moreover, we see that, for all $z \in \Lambda \cap B_\gamma(x)$,
  \begin{equation*}
    B_\delta(x) \cap B_1(x+z+(h(z)+1) \nu) \subseteq \{ u > 0 \} \subseteq \R^d \setminus B_1(x+z+(h(z)-1) \nu).
  \end{equation*}
  In particular, $h \in C^{1,1}(\Lambda \cap B_\gamma(x))$ and therefore $\nabla h = 0$.  Since $h(x) = 0$, we see that $\partial \{ u > 0 \}$ contains $\Lambda \cap B_\gamma(x)$.
  
  Step 4.
  We use the modified subsolution property again to obtain a contradiction.  Let $\ell \subseteq \Lambda \cap \partial \{ u > 0 \}$ be the maximal open line segment containing $x$.  We know from the previous step that $\ell \cap \partial \{ v > 0 \}$ is a compact subset of $\ell$.  We know that, for every $z \in \ell$, that $\nabla u(z)$ exists and satisfies $\nabla u(z) = |\nabla u(z)| \nu$ and $H(\nabla u(z)) \leq 1$.  Let $p = (\sup_{z \in \ell} |\nabla u(z)|) \nu$.  For any $\ep > 0$, there is an open $V \supseteq \ell \cap \partial \{ v > 0 \}$ such that the function $\varphi(y) = (1 + \ep) \tau p \cdot (y - x)$ touches $v$ from above in $V \cap \overline{\{ v > 0 \}}$.  By the modified subsolution condition, $H((1 + \ep) \tau p) = (1 + \ep) \tau H(p) \geq 1$.  This contradicts $H(p) \leq 1$ for $\ep > 0$ small enough.
\end{proof}

\subsection{Limits of half spaces}

In order for the notion of modified subsolution to be useful, we need to able to verify it for limits of the minimal supersolutions of \eref{introJh}.  This requires a more detailed analysis of the halfspace problem \eref{halfspace}.  We introduce another half space type problem which arises when we study the limiting behavior of $H(q)$ as $q \to p \neq 0$.  

We need to classify the collection of sets that can arise as limits of discrete half-spaces $E(\zeta) = \{x\in \Z^d: x \cdot \zeta >0\}$.  Let $\zeta^1,\dots,\zeta^k$ be an orthogonal set of non-zero vectors in $\R^d$ with $1 \leq k \leq d$.  We define the half-space type domain,
\begin{equation*}
  E(\zeta^1,\dots,\zeta^k) = \bigcup_{j=1}^k\{z \in \Z^d: [\zeta^1, \cdots ,\zeta^{j-1}]^tz = 0 \ \hbox{ and } \ \zeta^j \cdot z >0\}.
\end{equation*}
we interpret the first condition for $j=1$ as trivially satisfied.  We remark that the domain $E(\zeta^1,\dots,\zeta^k)$ is only distinct from $\{ z \cdot \zeta^1>0\} = E(\zeta^1)$ when the orthogonal complement of $\zeta^1$ contains at least one lattice vector.  That is, when $\Lambda_{\zeta^1} \neq \{ 0 \}$.  It is useful to note that,
\begin{equation}\label{e.limit hs induct}
 E(\zeta^1,\dots,\zeta^k) = E(\zeta_1) \cup (E(\zeta^2,\dots,\zeta^k) \cap  \{ \zeta^1 \cdot z =0\}),
 \end{equation}
which gives an inductive way to understand the limit half-spaces.

Now we study the linearly growing harmonic functions in the limit half-spaces.  Let $u_{\zeta^1 \cdots \zeta^k}$ denote the unique solution of
\begin{equation}
  \label{e.spec-half-space}
  \begin{cases}
    \Delta u_{\zeta^1\cdots\zeta^k} = 0 & \mbox{on } E(\zeta^1,\dots, \zeta^k) \\
    u_{\zeta^1\cdots\zeta^k} = 0 & \mbox{on } \Z^d \setminus E(\zeta^1,\dots, \zeta^k) \\
    \sup_{\zeta^1 \cdot x > 0} |u_{\zeta^1\cdots\zeta^k}(x) - x\cdot \zeta^1| < \infty.
  \end{cases}
\end{equation}

To explain the definition of the limit half-spaces we show how such domains arise naturally as the limit of standard half-spaces.  We consider the following sequence of half-spaces,
\[ E_t = E(\zeta^0+t \zeta^1 + t^2 \zeta^2 + \cdots + t^k \zeta^k) \ \hbox{ in the limit } \ t \to 0.\]
The particular powers of $t$ are not important, only that each subsequent term is lower order than the previous one as $t \to 0$.  Then,
\[ \lim_{t \to 0 } E_t = \{ z \in \Z^d: z \in E_t \ \hbox{ for all sufficiently small $t>0$}\} = E(\zeta^0,\dots, \zeta^k).\]
This is basically the way that the limit half spaces appear.

\begin{lemma}
  \label{l.halfspacelimits}
  Suppose that $\Sigma$ is a subspace of $\R^d$ and $\zeta_n \in \Sigma \cap S^{d-1}$ then, up to taking a subsequence, there exist $\zeta^1,\dots, \zeta^k \in \Sigma$ orthonormal for some $1 \leq k \leq \dim(\Sigma)$ so that,
\[ E(\zeta_n) \to E(\zeta^1,\dots,\zeta^k) \ \hbox{ as } \ n \to \infty.\]
\end{lemma}

\begin{proof}
Take a subsequence of the $\zeta_n$ converging to $\zeta^1 \in \Sigma \cap S^{d-1}$.  If $\zeta^n = \zeta_1$ for infinitely many $n$ we are done, taking a subsequence along which $\zeta_n$ are constant,
\[ E(\zeta_n) = E(\zeta_1).\]
Suppose instead that, eventually, $\zeta_n \neq \zeta^1$.  We define the approach direction of $\zeta_n$ to $\zeta$
\[ \zeta_n^2 = \frac{ \zeta_n - \zeta^1}{|\zeta_n - \zeta^1|} \ \hbox{ and we have } \ \zeta_n^2 \cdot \zeta^1 \to 0 \ \hbox{ as } \ n \to \infty.\]
Take a further subsequence so that $\zeta_n^2$ converges to some $\zeta^2 \in S^{d-1}$ with $\zeta^2 \cdot \zeta^1 = 0$.  Note that $\zeta^2_n \in \Sigma$ for all $n$ and hence so is $\zeta^2$.  Now we continue inductively, if $j=\dim(\Sigma)$ or if $\zeta_n^j = \zeta^j$ for infinitely many $n$ then we stop the process, taking if necessary another sub-sequence so that $\zeta_n^j = \zeta^j$ for all $n$.  Otherwise we can define,
\[ \zeta^{j+1}_n = \frac{ \zeta^{j}_n - \zeta^{j}}{|\zeta^{j}_n - \zeta^{j}|} \ \hbox{ and } \  \zeta_n^{j+1} \cdot \zeta^\ell  \to 0 \ \hbox{ as } \ n \to \infty \ \hbox{ for } \ 1 \leq \ell \leq j.\]
We check the orthogonality conditions by induction.  Thus we obtain an orthonormal set of vectors $\zeta^1,\zeta^2,\dots,\zeta^{k}$ for some $1 \leq k \leq \dim(\Sigma)$.  By the set-up we have either $(1)$ for all $n$ it holds $\zeta_n^k = \zeta^k$ or $(2)$ $k=\dim(\Sigma)$ and $\zeta^1,\dots,\zeta^k$ are an orthonormal basis for $\Sigma$.  

We claim now that,
\[ E(\zeta_n) \to E(\zeta^1,\dots,\zeta^k) \ \hbox{ as } \ n \to \infty.\]
If $z \cdot \zeta^1>0$ then $z \cdot \zeta^1_n>0$ eventually and if $z \cdot \zeta^1 <0$ then $z \cdot \zeta^1_n<0$ eventually.  
If $z \cdot \zeta^1 =0$ then,
\[  z\cdot\zeta^1_n >0 \iff z\cdot(\zeta^1_n-\zeta^1) >0 \iff z \cdot \zeta_n^2 >0.\]
So we have, in the case when $z \cdot \zeta^1 = 0$,
\[ z \cdot \zeta^2 >0 \ (\hbox{resp. $<0$}) \implies \ z \cdot \zeta_n >0 \ (\hbox{resp. $<0$})\ \hbox{ eventually.}\]
Suppose that $z \cdot \zeta^\ell = 0$ for $\ell=1,\dots, j$, then when is $\zeta^j_n \cdot z >0$ eventually,
\[ z \cdot \zeta^j_n >0 \iff z \cdot (\zeta^j_n- \zeta^j) >0 \iff z \cdot \zeta^{j+1}_n >0\]
So when $z \cdot \zeta^\ell = 0$ for $\ell=1,\dots, j$,
\[ z \cdot \zeta^{j+1} >0 \ (\hbox{resp. $<0$}) \implies \ z \cdot \zeta_n >0 \ (\hbox{resp. $<0$})\ \hbox{ eventually.}\]  
The final case is if $z \cdot \zeta^\ell = 0$ for all $\ell=1,\dots,k$.  If $k=\dim(\Sigma)$ then since $\zeta^1,\dots \zeta^k$ are an orthonormal basis for $\Sigma$ this holds only when $z  \in \Sigma^\perp$ in which case we know $\zeta_n \cdot z = 0$ for all $n$. Otherwise $\zeta_n^k = \zeta^k$ for all $n$ and so $z \cdot \zeta^k = 0$ if and only if $z \cdot \zeta^k_n = 0$.
\end{proof}

\begin{theorem}
  \label{t.limitsolutions}
  For each $\zeta^1,\dots, \zeta^k$ an orthogonal collection of non-zero vectors $1 \leq k \leq d$, there exists a unique solution $u_{\zeta^1\cdots\zeta^k}$ of the half-space problem \eref{spec-half-space} with the following properties,
  \begin{enumerate}
  \item \label{p1} 
    \[ \max_{x \in \partial^+ E(\zeta^1,\dots,\zeta^k)} \Delta u_{\zeta^1\cdots\zeta^k}(x) = \Delta u_{\zeta^1\cdots\zeta^k}(0)\]
  \item \label{p2} Let $\Sigma$ be the subspace spanned by $\zeta^1,\dots,\zeta^k$ then,
    \[\Delta u_{\zeta^1\cdots\zeta^k}(0) \leq \limsup_{\substack{p \in \Sigma \\ p \to \zeta^1}} H(p) \]
  \item \label{p3} If $\zeta^1$ is contained in the subspace $\Sigma'$, then there are $\xi^2,\cdots,\xi^k$ completing $\zeta^1$ to be an orthogonal basis of $\Sigma'$ such that,
    \[ \Delta u_{\zeta^1\xi^2\cdots\xi^k}(0) = \limsup_{\substack{p \in \Sigma' \\ p \to \zeta^1}} H(p) \]
  \end{enumerate}
\end{theorem}

\begin{proof}
Take a sequence,
\[ \zeta_n = \zeta^1 + \frac{1}{n} \zeta^2+\cdots + \frac{1}{n^k} \zeta^k \in \Sigma.\]
Clearly the $\zeta_n$ converge to $\zeta^1$ as $n \to \infty$.  Consider the half-space solutions $u_{\zeta_n}$ solving \eref{spec-half-space} in $E(\zeta_n)$.  Since we have shown in \lref{halfspacelimits} that
\[ E(\zeta_n) \to E(\zeta^1,\dots,\zeta^k) \ \hbox{ as } \ n \to \infty.\]
We can establish, by uniqueness of \eref{spec-half-space}, that,
\[ u_{\zeta_n} \to u_{\zeta^1\cdots \zeta^k} \ \hbox{ as } \ n \to \infty.\]
In particular,
\[ H(\zeta_n) = \Delta u_{\zeta_n}(0) \to \Delta u_{\zeta^1\cdots \zeta^k}(0),\]
which proves part~\ref{p2}.

Next we consider the proof of part~\ref{p3}.  There is a sequence of $\zeta_n \in \Sigma'$ such that $\zeta_n \to \zeta^1$ and $H(\zeta_n) \to \limsup_{p \in \Sigma' \to \zeta^1} H(p)$.  Applying \lref{halfspacelimits} we find $\xi^2,\dots,\xi^k$ for some $1 \leq k \leq \dim (\Sigma)$ such that, up to a subsequence, the $u_{\zeta_n}$ converge and,
\[ E(\zeta_n) \to E(\zeta^1,\xi^2,\dots,\xi^k) \ \hbox{ as } \ n \to \infty.\]
Then, as usual, the limit of the $u_{\zeta_n}$ solves \eref{spec-half-space} for $\zeta^1,\xi^2,\dots,\xi^k$.  By uniqueness the limit is $u_{\zeta^1\xi^2\cdots \xi^k}$ and we obtain 
\[\limsup_{p \in \Sigma' \to \zeta^1} H(p) = \lim_n H(\zeta_n) = \lim_n \Delta u_{\zeta_n}(0) = \Delta u_{\zeta^1,\xi^2, ..., \xi^k}(0).\]
\end{proof}

\subsection{Sphere approximation}  

We now explain how to construct a discrete approximation $\Lambda_m$ of the lattice ball $B(0,m)$ with a certain regularity property.  Basically the regularity property we would like to have is that, when we send $m \to \infty$ and follow a sequence of boundary points $z_m$ the blow up $\Lambda_m - z_m$ will converge (along a subsequence) to the complement of a limit half space.  Once we have done this construction for balls we will also be able to construct regular discrete domains approximating any smooth domain in $\R^d$ by an inf-convolution type argument.

We will use the following notation below, for a collection $e_1,\dots,e_k \subset \R^d$,
\[ T(e_1,\dots,e_k) = \{ x \in \R^d: \  x \cdot e_1 = \cdots = x\cdot e_k = 0\}.\]

\begin{lemma}\label{l.discreteball}
Let $\Sigma$ be a $k$-dimensional rational subspace of $\R^d$ and let $r \geq 1$.  There exists a sequence of sets $\Lambda_{m}(\Sigma) \subset \Z^d \cap \Sigma$ with $\Lambda_{m} = - \Lambda_{m}$ and,
\[ B(0,m) \cap \Sigma \subset \Lambda_{m}  \subset B(0,m+1) \cap \Sigma , \]
 with the following property: for any sequence $z_m \in  \Lambda_{m}$ such that $\Lambda_{m} - z_m$ converges to a set $\mathcal{H}$ then there is some orthonormal set $\zeta^1,\dots,\zeta^k \in \Sigma \cap T(p)$ for $1 \leq k \leq \dim(\Sigma)$ such that,
\[ \mathcal{H}^C \cap \Sigma \cap B(0,r) \subset E(\zeta^1,\dots,\zeta^k).\]
\end{lemma}

\begin{proof}
Instead of working in a subspace $\Sigma$ of $\R^d$ with the lattice $\Z^d \cap \Sigma$, we will just work in $\R^d$ with an arbitrary $d$-dimensional integer lattice $\mathcal{Z}^d$.  We denote, for $r>0$ and $x \in \R^d$, the ball of radius $r$ in $\mathcal{Z}^d$ around $x$ by $\mathcal{B}(x,r) = \{ z \in \mathcal{Z}^d : |z-x| \leq r\}$.  

The proof is by an induction argument on the dimension.  We start with the case $d=1$.  If $\mathcal{Z}$ is an integer lattice in $\R$ then,
\[ \mathcal{Z} = \{n e: n \in \Z\} \ \hbox{ for some } \ e \in \R \setminus \{0\}.\]
It will be convenient for the induction to define the regularized balls $\Lambda_m$ around arbitrary points, not just lattice points.  For $x \in \R$ and $m\geq 1$ we can simply take 
\[\Lambda_m(x) = \mathcal{B}(x,m).\]
  Then, for a sequence $z_m \in  \Lambda_m(x_m)$, we take a subsequence so that $\Lambda_m(x_m) - z_m$ converges to some $\mathcal{H}$, and then,
\[ \mathcal{H} \supset   \{ ne : n \leq 0 \} \ \hbox{ or } \ \mathcal{H} \supset \{ ne : n \geq 0 \},\]
which is what we required.

Now we assume that the result is proven up to dimension $d-1$, and we consider a $d$-dimensional integer lattice $\mathcal{Z}^d$ in $\R^d$.  Let $m \in \mathbb{N}$, start with the sets $\mathcal{B}(0,m) = \{ z \in \mathcal{Z}^d : |z| \leq m\}$. We fix an $r\geq 1$ and define for $1 \leq k \leq d$
\[ S_k(r) = \{ p \in \R^d \setminus \{0\}:  \hbox{ the dimension of } \operatorname{span}(T(p) \cap \mathcal{B}(0,r)) \ \hbox{ is at least $d-k$} \}, \]
these sets are evidently monotone increasing in $r$ and in $k$ and $S_d(r) = \R^d \setminus \{0\}$.  In words, $S_k(r)$ is the set of directions satisfying $k$ independent rational relations with norm at most $r$ with respect to the lattice $\mathcal{Z}^d$.  Let $\ell_k \geq 1$ chosen as $\ell_0 = m$ and for $1 \leq k \leq d-1$,
\begin{equation}\label{e.ldef}
 \ell_k(m) = m^{\frac{1}{4^{k}}} \quad  \hbox{ chosen to satisfy } \quad  \ell_{k+1}^2 = o(\ell_{k}).
 \end{equation}
 Given $p \in \R^d$ the lattice $\mathcal{Z}^d \cap T(p)$ has dimension at most $d-1$. For $x \in T(p)$ and $\ell \geq 1$ let $\Lambda_\ell(x,\mathcal{Z}^d \cap T(p))$ be the ball approximations which are already defined by the inductive hypothesis.  Furthermore, every affine hyperplane $x + T(p)$ for $x \in \R^d$ either contains no points of $\mathcal{Z}^d$ or contains a translated copy of $\mathcal{Z}^d \cap T(p)$.  Thus we can define for every $x \in \R^d$ the set $\Lambda_\ell(x,\mathcal{Z}^d \cap T(p)) \subset (x+T(p)) \cap \mathcal{Z}^d$.  
 
 Now we are ready to define the sets $\Lambda_m$ in $\mathcal{Z}^d$.   We define
\begin{equation}\label{e.lambdadef}
 \Lambda_{m} = \mathcal{B}(0,m) \cup \bigcup_{k=1}^{d-1}\bigcup_{p \in S_k(r) \cap B(0,m)} \Lambda_{\ell_k}(p,\mathcal{Z}^d \cap T(p)).
 \end{equation}
 For the ball around a generic point $x \in \R^d$ the definition is analogous.  The set $\Lambda_m$ has the following property, for every $z \in \Lambda_m$ either there is $1 \leq k \leq d-1$ such that,
\begin{equation}\label{e.zk}
 z \in \Lambda_{\ell_k}(p, \mathcal{Z}^d \cap T(p)) \ \hbox{ for some } \ p \in S_k(r)\cap B(0,m),
 \end{equation}
or $z \in \mathcal{B}(0,m)$.  Let us define $k(z)$ to be the minimum value of $k$ so that \eref{zk} holds, or if \eref{zk} does not hold for any $1\leq k \leq d-1$ then $z \in \mathcal{B}(0,m)$ and we take $k(z)=d$.  We define $p(z)$ to be the $p \in S_{k(z)}(r)$ achieving \eref{zk}, taking $p(z) = z$ if $k(z) = d$.  We note that,
\begin{equation}
\hbox{ for any } j < k(z) \ \hbox{ we have } \ d(z,S_j(r)) >\ell_j.
\end{equation}
We also point out that for any $z \in \Lambda_m$,
\[ |z|^2 = |z-p(z)|^2+|p(z)|^2 \leq (\ell_1+1)^2+m^2 \ \hbox{ and so } \ |z| \leq m+\frac{1}{2}\frac{(\ell_1+1)^2}{m^2} \leq m+1\]
for $m$ sufficiently large by the inductive hypothesis and since we chose $\ell_1^2 = o(m)$.

    We consider a sequence $z_{m} \in \Lambda_m$.   Up to taking a subsequence we can assume that $k = k(z_m)$ is constant, we call $p_m = p(z_m)$.    By it's definition $p_m \in S_k(r) \cap B(0,m)$ with,
  \[ z_m  \in \Lambda_{\ell_k}(p_m,\mathcal{Z}^d \cap T(p_m)),\]
    in the case $k=d$ we have $p_m = z_m \in \mathcal{B}(0,m)$. We can choose a further subsequence so that for some fixed $\xi_1,\dots, \xi_{d-k} \in \mathcal{B}(0,r)$,
  \[  p_m \in \Xi^\perp,\]
  where call $\Xi$ to be (interchangeably) the subspace spanned by $\xi_1,\dots,\xi_{d-k}$ and the corresponding integer lattice.  We remark that in the case $k=d$ this is a trivial condition.     We now take a further subsequence if necessary so that,
  \[\Lambda_{m} - z_m \to \mathcal{H} \subset \Z^d \cap \Sigma \ \hbox{ and } \ \frac{p_m}{|p_m|} \to p \in S^{d-1} .\]    
It is easy to see that $ p \in \Xi^\perp$, hence $p \in S_k(r)$, although it may satisfy additional rational relations in $\mathcal{B}(0,r)$.

 Now we show that the limit set $\mathcal{H}$ has the structure of a limit half-space at least until we get down to the level of $\Xi$, this part of the proof will be familiar from \lref{halfspacelimits}, although it contains some new elements due to the curvature of the ball.  Call $\zeta^1_m = p_m$ and $\zeta^1 = p$.  We will define a set of mutually orthogonal directions $\zeta^j \in \Xi^\perp$ for (at most) $1 \leq j \leq k$ by an inductive procedure, these directions will be defined as the limit of some sequences $\hat{\zeta}_m^j$.  If at any stage $1 \leq j \leq k$,
\[ T(\zeta^1,\dots,\zeta^j) \cap \mathcal{B}(0,r)  \subset \Xi\]
then we will stop the induction.  Note that since $\Xi$ has dimension $d-k$, $\Xi \subset T(\zeta^1,\dots,\zeta^j)$, and the dimension of $T(\zeta^1,\dots, \zeta^j)$ is $d-j$ the induction must stop by stage $j=k$.  Suppose that $\dim(T(\zeta^1,\dots,\zeta^{j-1}) \cap \mathcal{B}(0,r))>n-k$, then the induction can continue and we define, 
\[ \zeta^j_m = \zeta^{j-1}_m - |\zeta^{j-1}_m| \zeta^{j-1} \ \hbox{ or expanded } \ \zeta^j_m = p_m -\sum_{i=1}^{j-1} |\zeta^i_m|\zeta^i = p_m - q^j_m. \]  
From the inductive assumption, as $q^j_m$ is a linear combination of $\zeta^1,\dots, \zeta^{j-1}$, we have $q^j_m \in \Xi^\perp$ and $\zeta^j_m \in \Xi^\perp$.  Since  $T(\zeta^1,\dots, \zeta^{j-1}) \subset T(q^j_m)$ and $\dim(T(\zeta^1,\dots,\zeta^{j-1}) \cap \mathcal{B}(0,r))>n-k$ we have that $q^j_m \in S_{k-1}(r)$.  From our assumptions \eref{sk} this means that,
\[ |\zeta^j_m| = |p_m - q^j_m| > \ell_{k-1}. \]
In particular it is not equal to $0$.  Then we take a subsequence so that $\hat{\zeta}^j_m$ converges to some $\zeta^j \in S^{d-1} \cap \Xi^\perp$.  It is also easily checked that $\zeta^1,\zeta^2,\dots,\zeta^{j-1}$ is an orthonormal set.

Now given the above set up we are able to establish the limit set structure. Suppose that,
 \begin{equation}\label{e.limitcond}
  z \cdot \zeta^i = 0 \ \hbox{ for } \ 1 \leq i \leq j-1 \ \hbox{ and } \ z \cdot \zeta^j <0
  \end{equation}
 then we claim that $z \in \mathcal{H}$.  For this we show that $|z_m+z| \leq m$ eventually,
   \begin{align*}
   |z_m+z|^2 &\leq |p_m|^2  +2|z_m - p_m|^2  +2|z|^2+2 z \cdot p_m \\
   &\leq m^2 +2\ell_k^2+2|z|^2 +2 z \cdot (p_m -  q^j_m) \\
   &\leq m^2 +2\ell_k^2+2|z|^2+ (z \cdot \hat{\zeta}^j_m) \ell_{k-1} \\
   & <m^2
   \end{align*}
   since we have by the set up that $\ell_k^2 = o(\ell_{k-1})$, i.e. under \eref{limitcond} we have $z \in \mathcal{H}$.

Now we suppose that the inductive procedure above ends at some stage $1 \leq j_0 \leq k$, the ending condition is that $T(\zeta^1,\dots,\zeta^{j_*}) \cap \mathcal{B}(0,r) = \Xi \cap \mathcal{B}(0,r)$.  Now we use the main induction hypothesis.  We already have that,
\[ z_m \in \Lambda_{\ell_k}(p_m,\mathcal{Z}^d \cap T(p_m)) \subset \Lambda_m.\]
Since $\Lambda_{\ell_k}(p_m,\mathcal{Z}^d \cap T(p_m)) \cap B(z_m,r) = \Lambda_{\ell_k}(p_m,\Xi)$ for the fixed subspace $\Xi$ we can use the inductive hypothesis.  Taking a subsequence so that $\Lambda_{\ell_k}(p_m,\Xi) - z_m$ converges to some set $\mathcal{G} \subset \Xi$, we have $\mathcal{G} \subset \mathcal{H}$, and the inductive hypothesis implies that there are some orthonormal $\zeta^{j_0+1},\dots,\zeta^{j_1} \in \Xi$ so that,
\[ \mathcal{G}^C \cap T(\zeta^1,\dots,\zeta^{j_0})  \cap B(0,r) \subset E(\zeta^{j_0+1},\dots,\zeta^{j_1}).\]
 Thus,
 \[ B(0,r) \cap \bigcup_{i=1}^{j_1} \{z \in \mathcal{Z}^d: \  z \in T(\zeta^1,\dots,\zeta^{i-1}) \cap \{ z \cdot \zeta^i <0\} \} \subset \mathcal{H}.   \]
 This completes the proof.
\end{proof}

\subsection{Perturbed test functions}

We now explain how to construct discrete approximations of a set of classical supersolutions.  This suffices to verify that limits of local minimizers of \eref{introJh} are modified subsolutions.  We need to choose a sequence of discrete approximation to the classical supersolution such that all the boundary blow up limits are limit half-spaces.  The main difficulty was already included in \lref{discreteball}, now we use the ball approximations there to regularize a general discrete set by inf-convolution.

\begin{theorem}
  \label{t.ptfm}
  Suppose that $U \subseteq \R^d$ is open and bounded, $\varphi \in C^\infty(U)$, $L \varphi \leq 0$ in $\{ \varphi > 0 \}$, $\Sigma \subseteq \R^d$ is a subspace, $\nabla \varphi \subseteq \Sigma$, and $\lim_{p \in \Sigma \to \nabla \varphi(x)} H(p) \leq 1$ whenever $\varphi(x) = 0$.  If $V \subseteq U$ open, $\overline V \subseteq U$, and $h_n \to \infty$, we can find $u_n : \Z^d \to \R$ such that $u_n \geq 0$, $\Delta u_n \leq (1 + o(1)) \id_{\{ u_n = 0 \}}$, and $h_n^{-1} u_n(h_n x) = \max \{ \varphi(x), 0 \} + o(1)$ hold in $h_n V$.
\end{theorem}

\begin{proof}
  Step 1.
  Throughout the proof, we let $C$ denote a positive constant that depends on $\varphi, V, U$ and may differ in each instance.  For integers $n, m > 1$, define
  \begin{equation*}
    \varphi_n(x) = h_n \varphi(h_n^{-1} x)
  \end{equation*}
  and
  \begin{equation*}
    \varphi_{n,m}(x) = \min_{y \in \Lambda_m(\Sigma)} \varphi_n(x+y)
  \end{equation*}
  where $\Lambda_m(\Sigma)$ is given by \lref{discreteball}.  Choose $W \supseteq \overline V$ open with $\overline W \subseteq U$ and let $\psi_{n,m}$ solve
  \begin{equation*}
    \begin{cases}
      \Delta \psi_{n,m} = \min \{ 0, \Delta \varphi_{n,m} \} & \mbox{in } h_n W \cap \{ \varphi_{n,m} > 0 \} \\
      \psi_{n,m} = \max \{ 0, \varphi_{n,m} \} & \mbox{in } h_n U \setminus (h_n W \cap \{ \varphi_{n,m} > 0 \})
    \end{cases}
  \end{equation*}
  Since $\varphi$ is smooth and $L \varphi \leq 0$, we have $-Ch_n^{-1}\leq \Delta \varphi_n \leq C h_n^{-3}$ in $h_n U$ and $-Ch_n^{-1} \leq \Delta \varphi_{n,m} \leq C h_n^{-3}$ in $h_n W$.  Moreover, we also have that $\varphi_n$ and $\varphi_{n,m}$ are uniformly Lipschitz.  In particular, we see that
  \begin{equation*}
   \varphi_{n,m} \leq \psi_{n,m} \leq \varphi_{n,m}+ Ch_n^{-1} \quad \mbox{in } h_n W \cap \{ \varphi_{n,m} > 0 \}.
  \end{equation*}
  and
  \begin{equation*}
    0 \leq \varphi_n - \varphi_{n,m} \leq C_m \quad \mbox{in } h_n W.
  \end{equation*}
  It suffices to show that, for all $\ep > 0$, there are $N, m > 1$ such that
  \begin{equation*}
    \Delta \psi_{n,m} \leq (1 + \ep) \id_{\{ \psi_{n,m} = 0 \}} \quad \mbox{in } h_n V
  \end{equation*}
  holds for all $n \geq N$.  We suppose for contradiction that this fails and run blow up argument.  We will first send $n \to \infty$ and then send $m \to \infty$.

  Step 2.
  Fix $\ep > 0$ and $m > 1$ and suppose, for infinitely many $n > 1$, there is a $z_n \in h_n V \cap \partial^+ \{ \psi_{n,m} > 0 \}$ such that
  \begin{equation*}
    \Delta \psi_{n,m}(z_n) \geq 1 + \ep.
  \end{equation*}
  Since
  \begin{equation*}
    h_n W \cap \{ \psi_{n,m} > 0 \} = h_n W \cap \{ \varphi_{n,m} > 0 \}
  \end{equation*}
  we may select a $w_n \in \partial^+ \{ \varphi_n > 0 \}$ such that
  \begin{equation*}
    z_n \in w_n + \Lambda_m(\Sigma) \subseteq \{ \psi_{n,m} = 0 \}.
  \end{equation*}
  Let $p_n = \nabla \varphi_n(w_n) = \nabla \varphi(h_n^{-1} w_n)$.  Taylor expanding $\varphi_n$ at $w_n$, we see that
  \begin{equation*}
    \varphi_{n,m}(z_n + z) \leq \varphi_n(w_n + z) \leq p_n \cdot z + C h_n^{-1} |z|^2
  \end{equation*}
  and
  \begin{equation*}
    \varphi_{n,m}(z_n + z) \geq \varphi_n(w_n + z) - C_m \geq p_n \cdot z - C h_n^{-1} |z|^2 - C_m.
  \end{equation*}
  Since $|\psi_{n,m} - \varphi_{n,m}| \leq C$, we may choose a sequence $n \to \infty$, $\bar z_m \in \overline V$, $\bar w_m \in \Z^d$, and $\psi_m : \Z^d \to \R$ such that
  \begin{equation*}
    h_n^{-1} z_n \to \bar z_m,
    \quad
    p_n \to \bar p_m = \nabla \varphi(\bar z_m),
    \quad
    w_n - z_n \to \bar w_m,
    \quad \mbox{and} \quad
    \psi_{n,m}(z_n+z) \to \psi_m(z).
  \end{equation*}

  Since
  \begin{multline*}
      \{ \psi_{n,m}(z_n + z) = 0 \} = \{ \varphi_{n,m}(z_n + z) \leq 0 \} \\ \supseteq (w_n - z_n + \Lambda_m(\Sigma)) \cup \{ p_n \cdot z + C h_n^{-1} |z|^2 \leq 0 \},
  \end{multline*}
  we see that
  \begin{equation*}
    \{ \psi_m(z) = 0 \} \supseteq \{ \bar p_m \cdot z < 0 \} \cup (\bar w_m + \Lambda_m(\Sigma)).
  \end{equation*}
  Note the strict inequality that arises because we do not know the relative rates of $p_n \to p$ and $h_n^{-1} \to 0$ as $n \to \infty$.  We also have
  \begin{equation*}
    0 \in \bar w_m + \Lambda_m(\Sigma)
  \end{equation*}
  \begin{equation*}
    \Delta \psi_m(0) \geq 1 + \ep
  \end{equation*}
  and
  \begin{equation*}
    \Delta \psi_m = 0
    \quad \mbox{and} \quad
    |\psi_m(x) - \bar p_m \cdot x| \leq C_m
    \quad \mbox{in }
    \{ \psi_m > 0 \}.
  \end{equation*}

  Step 3.  
  Observe that, since $\{ \bar p_m \cdot z < 0 \} \subseteq \{ \psi_m(z) = 0 \}$ and $\bar p_m$ is bounded uniformly in $m$, that we have
  \begin{equation*}
    0 \leq \psi_m(x) \leq \max \{ 0, \bar p_m \cdot x \} + C.
  \end{equation*}
  Thus we again have compactness of the $\psi_m$ and may choose a sequence $m \to \infty$, $\bar z \in \bar V$, $E \subseteq \Z^d$, and $\psi : \Z^d \to \R$ such that
  \begin{equation*}
    \bar z_m \to \bar z,
    \quad
    \bar p_m \to \bar p = \nabla \varphi(\bar z),
    \quad
    \bar w_m + \Lambda_m(\Sigma) \to E
    \quad \mbox{and} \quad
    \psi_m(z) \to \psi(z).
  \end{equation*}
  Observe that
  \begin{equation*}
    \{ \bar p \cdot z < 0 \} \cup E \subseteq \{ \psi = 0 \},
  \end{equation*}
  \begin{equation*}
    \Delta \psi = 0 \quad \mbox{and} \quad |\psi(x) - p \cdot x| \leq C \quad \mbox{in } \{ \psi > 0 \}
  \end{equation*}
  and
  \begin{equation*}
    \psi(0) = 0 \quad \mbox{and} \quad \Delta \psi(0) \geq 1 + \ep.
  \end{equation*}
  Since $\nabla \varphi \subseteq \Sigma$, we also have
  \begin{equation*}
    \psi_m(z + q) = \psi_m(z) \quad \mbox{for } q \in \Z^d \cap \Sigma^\perp.
  \end{equation*}
  Thus, by \lref{discreteball}, we see that there are $\zeta^2, ..., \zeta^k \in \Sigma$ such that
  \begin{equation*}
    E(\bar p, \zeta^2, ..., \zeta^k) \subseteq \{ \psi = 0 \}.
  \end{equation*}
  By the maximum principle, we see that
  \begin{equation*}
    \Delta u_{\bar p, \zeta^2, ..., \zeta^k}(0) \geq \Delta \psi(0) \geq 1 + \ep.
  \end{equation*}
  In particular, we have $\limsup_{\Sigma \ni p  \to \nabla \varphi(\bar z)} H(p) \geq 1 + \ep$, contradicting the theorem hypothesis.
\end{proof}


\begin{bibdiv}
  \begin{biblist}
  
\bib{Armstrong-Sirakov-Smart}{article}{
author={Armstrong, Scott N.},
author={Sirakov, Boyan},
author={Smart, Charles K.},
title={Singular Solutions of Fully Nonlinear Elliptic Equations and Applications},
journal={Archive for Rational Mechanics and Analysis},
year={2012},
volume={205},
number={2},
pages={345--394},
}
  
\bib{Alberti-DeSimone}{article}{
	Author = {Alberti, Giovanni},
	author={DeSimone, Antonio},
	Journal = {Proc. R. Soc. Lond. Ser. A Math. Phys. Eng. Sci.},
	Number = {2053},
	Pages = {79--97},
	Title = {Wetting of rough surfaces: a homogenization approach},
	Volume = {461},
	Year = {2005}
	}


    \bib{Alt-Caffarelli}{article}{
      author={Alt, H. W.},
      author={Caffarelli, L. A.},
      title={Existence and regularity for a minimum problem with free boundary},
      journal={J. Reine Angew. Math.},
      volume={325},
      date={1981},
      pages={105--144},
    }

    \bib{Aleksanyan-Shahgholian}{article}{
      author={Aleksanyan, H.},
      author={Shahgholian, H.},
      title={Discrete Balayage and Boundary Sandpile},
      eprint={ArXiv:1607.01525}
    }

    \bib{Aleksanyan-Shahgholian-2}{article}{
      author={Aleksanyan, H.},
      author={Shahgholian, H.},
      title={Perturbed divisble sandpiles and quadrature surfaces},
      eprint={ArXiv:1703.07568}
    }
    
    \bib{Caffarelli-Lee}{article}{
      author={Caffarelli, L.},
      author={Lee, K.},
      title={Homogenization of oscillating free boundaries: the elliptic case},
      journal={Comm. Partial Differential Equations},
      volume={32},
      date={2007},
      number={1-3},
      pages={149--162},
    }
    
    
\bib{Caffarelli-Mellet0}{article}{
	Author = {Caffarelli, L. A.},
	author = {Mellet, A.},
	Booktitle = {Perspectives in nonlinear partial differential equations},
	Pages = {175--201},
	Publisher = {Amer. Math. Soc., Providence, RI},
	Series = {Contemp. Math.},
	Title = {Capillary drops on an inhomogeneous surface},
	Volume = {446},
	Year = {2007}
	}


    \bib{Caffarelli-Mellet}{article}{
      author={Caffarelli, L. A.},
      author={Mellet, A.},
      title={Capillary drops: contact angle hysteresis and sticking drops},
      journal={Calc. Var. Partial Differential Equations},
      volume={29},
      date={2007},
      number={2},
      pages={141--160},
    }
    
    \bib{Caffarelli-Salsa}{book}{
      author={Caffarelli, Luis},
      author={Salsa, Sandro},
      title={A geometric approach to free boundary problems},
      series={Graduate Studies in Mathematics},
      volume={68},
      publisher={American Mathematical Society, Providence, RI},
      date={2005},
      pages={x+270},
    }

    \bib{Caffarelli-Spruck}{article}{
      author={Caffarelli, Luis A.},
      author={Spruck, Joel},
      title={Convexity properties of solutions to some classical variational
        problems},
      journal={Comm. Partial Differential Equations},
      volume={7},
      date={1982},
      number={11},
      pages={1337--1379},
    }

    \bib{Feldman}{article}{
      author={Feldman, William M.},
      title={Homogenization of the oscillating Dirichlet boundary condition in
        general domains},
      language={English, with English and French summaries},
      journal={J. Math. Pures Appl. (9)},
      volume={101},
      date={2014},
      number={5},
      pages={599--622},
    }
    
    

\bib{Feldman-Kim}{article}{
	Author = {William M. Feldman and Inwon C. Kim},
	archivePrefix = {arXiv},
        eprint = {ArXiv:1612.07261},
	Month = {12},
	Title = {Liquid Drops on a Rough Surface},
	Year = {2016}
	}

\bib{Kim-Zheng-Stone}{article}{
  title={Noncircular stable displacement patterns in a meshed porous layer},
  author={Kim, Hyoungsoo},
  author={Zheng, Zhong},
  author={Stone, Howard A},
  journal={Langmuir},
  volume={31},
  number={20},
  pages={5684--5688},
  year={2015},
  publisher={ACS Publications}
}

    \bib{Kim}{article}{
      author={Kim, Inwon C.},
      title={Homogenization of a model problem on contact angle dynamics},
      journal={Comm. Partial Differential Equations},
      volume={33},
      date={2008},
      number={7-9},
      pages={1235--1271},
    }
    
    

    \bib{Levine}{thesis}{
      author={Levine, L.},
      title={Limit Theorems for Internal Aggregation Models},
      organization={University of Californial Berkeley},
      note={PhD Thesis},
      date={2007}
    }

    \bib{Raj-Adera-Enright-Wang}{article}{
      author={Raj, R.},
      author={Adera, S.},
      author={Enright, R.},
      author={Wang, E.}
      title={High-resolution liquid patterns via three-dimensional droplet shape control},
      journal={Nature Communications},
      date={12 Aug 2014},
      doi={10.1038/ncomms5975}
    }
    
    \bib{pattern}{article}{
author = {Susarrey-Arce, A.},
author={Marin, A.},
 author={Massey, A.},
 author={Oknianska, A.},
author={D\'iaz-Fernandez, Y.},
author={Hern\'andez-S\'anchez, J. F.},
author={Griffiths, E.},
author={Gardeniers, J. G. E.},
author={Snoeijer, J. H.},
author={Lohse, Detlef},
author={Raval, R.},
title = {Pattern Formation by Staphylococcus epidermidis via Droplet Evaporation on Micropillars Arrays at a Surface},
journal = {Langmuir},
volume = {32},
number = {28},
pages = {7159-7169},
year = {2016},

}

  \end{biblist}
\end{bibdiv}
\end{document}